\documentclass[12pt,twoside]{article}
\usepackage[hmargin=0.8in,vmargin=0.9in]{geometry}
\usepackage{fancyhdr}
\usepackage{graphicx}
\usepackage{amssymb}
\usepackage{amsmath}
\usepackage{amsfonts}
\usepackage{theorem}
\usepackage{mathrsfs}
\usepackage{mathtools}
\usepackage{bm}
\usepackage{color}
\usepackage{nicefrac}
\usepackage{setspace}
\usepackage{exscale}
\usepackage{relsize}
\usepackage{txfonts}
\usepackage{epstopdf}
\usepackage{float}
\usepackage{cite}
\usepackage{hyperref}
\usepackage{cleveref}

\usepackage[absolute,overlay]{textpos}
\usepackage[utf8]{inputenc}
\usepackage[english]{babel}
\usepackage{caption}
\usepackage{subcaption}
\usepackage{pdfpages}
\usepackage{comment}
\usepackage{etoolbox}
\usepackage{tikz}
\usetikzlibrary{calc}

\usepackage{algpseudocode}
\usepackage{algorithm}
\usepackage{multirow}

\usepackage{pgfplots}
\usepackage{pgfplotstable}
\usepackage{stmaryrd}

\usetikzlibrary{fillbetween}
\usepgfplotslibrary{fillbetween}
\usetikzlibrary{shapes.misc}
\usetikzlibrary{math}
\usetikzlibrary{arrows}
\usetikzlibrary{patterns,patterns.meta}

\usepackage{array} 
\usepackage{paralist} 
\usepackage{verbatim} 

\theoremstyle{plain}                    
\newtheorem{thm}{Theorem}[section]

\newtheorem{proposition}[thm]{Proposition}

\newtheorem{remark}{Remark}[section]
\newenvironment{proof}{{\noindent\it Proof.}\quad}{\hfill $\square$\par}

\allowdisplaybreaks[1]

\numberwithin{equation}{section}
\numberwithin{figure}{section}
\numberwithin{table}{section}

\newcommand\eref[1]{(\ref{#1})}

\newcommand*\xbar[1]{%
  \hbox{%
    \vbox{%
      \hrule height 0.5pt 
      \kern0.4ex
      \hbox{%
        \kern-0.05em
        \ensuremath{#1}%
        \kern-0.00em
      }%
    }%
  }%
}

\newcommand{\uvec}[2][3]{\bm{#2\mkern-#1mu}\mkern#1mu}

\allowdisplaybreaks[1]

\newcommand{\dt}{\Delta t}
\newcommand{\dx}{\Delta x}
\newcommand{\dy}{\Delta y}

\newcommand{\hf}{{\frac{1}{2}}}

\newcommand{\jph}{{j+\frac{1}{2}}}
\newcommand{\jmh}{{j-\frac{1}{2}}}
\newcommand{\kph}{{k+\frac{1}{2}}}
\newcommand{\kmh}{{k-\frac{1}{2}}}

\newcommand{\Vbar}{\overline{\uvec{V}}}

\title{Dual Formulation Finite-Volume Methods on Overlapping Meshes for Hyperbolic Conservation Laws}
\author{R\'emi Abgrall\thanks{Institute for Mathematics \& Computational Science, University of Zurich, 8057 Zurich, Switzerland;
{\tt remi.abgrall@math.uzh.ch}}, ~Alina Chertock\thanks{Department of Mathematics, North Carolina State University, Raleigh, NC 27695, USA;
{\tt chertock@math.ncsu.edu}}, ~Alexander Kurganov\thanks{Department of Mathematics and Shenzhen International Center for Mathematics,
Southern University of Science and Technology, Shenzhen, 518055, China; {\tt alexander@sustech.edu.cn}},~ and Lorenzo
Micalizzi\thanks{Department of Mathematics, North Carolina State University, Raleigh, NC 27695, USA; {\tt lmicali@ncsu.edu}}}
\date{}
\begin{document}
\maketitle
\begin{abstract}
In this work, we introduce new second-order schemes for one- and two-dimensional hyperbolic systems of conservation laws. Following an
approach recently proposed in [{\sc R. Abgrall}, Commun. Appl. Math. Comput., 5 (2023), pp. 370--402], we consider two different
formulations of the studied system (the original conservative formulation and a primitive one containing nonconservative products), and
discretize them on overlapping staggered meshes using two different numerical schemes. The novelty of our approach is twofold. First, we
introduce an original paradigm making use of overlapping finite-volume (FV) meshes over which cell averages of conservative and primitive
variables are evolved using semi-discrete FV methods: The nonconservative system is discretized by a path-conservative central-upwind
scheme, and its solution is used to evaluate very simple numerical fluxes for the discretization of the original conservative system.
Second, to ensure the nonlinear stability of the resulting method, we design a post-processing, which also guarantees a conservative
coupling between the two sets of variables. We test the proposed semi-discrete dual formulation finite-volume methods on several benchmarks
for the Euler equations of gas dynamics.
\end{abstract}

\smallskip
\noindent
{\bf Key words:} Dual formulation finite-volume methods; overlapping staggered meshes; path-conservative central-upwind schemes;
conservative post-processing.

\medskip
\noindent
{\bf AMS subject classification:} 76M12, 65M08, 76N99, 35L65, 35L67.

\section{Introduction}
This paper focuses on the development of numerical methods for hyperbolic systems of conservation laws, which in the two-dimensional (2-D)
case read as
\begin{equation}
\bm U_t+\bm F(\bm U)_x+\bm G(\bm U)_y=\bm0.
\label{1.1}
\end{equation}
Here, $x$ and $y$ are the spatial variables, $t$ is time, $\bm U\in\mathbb R^M$ is the vector of conserved variables,
$\bm F,\bm G\in\mathbb R^M$ are the fluxes, whose Jacobians,
$\frac{\partial\bm F}{\partial\bm U},\frac{\partial\bm G}{\partial\bm U}\in\mathbb R^{M\times M}$, are assumed to be real-diagonalizable,
and $M\in\mathbb N$ with $M\ge1$ being the number of equations in the system.

It is well-known that solutions of nonlinear hyperbolic conservation laws may become nonsmooth even when the initial and boundary data are
infinitely smooth. Therefore, the solution of \eref{1.1} has to be defined in a weak (integral) sense, and hence conservative 
finite-volume (FV) methods seem to be one of the natural choices to be considered. In these methods, the computational domain is covered 
by FV cells and the numerical solution is realized in terms of cell averages of conserved variables, which are evolved in time using an
integral form of \eref{1.1}. For a variety of existing FV methods, we refer the reader to \cite{GR3,Hes,LeV02,ToroBook} and references
therein.

In this work, we are interested in schemes which make use of different formulations of the same governing equations, namely, conservative
and nonconservative (primitive) formulations. Examples of such schemes are the active flux (AF) schemes introduced in \cite{Abg23}, in which
cell averages of conserved variables and point values of primitive ones at cell interfaces are considered. Such additional degrees of
freedom with respect to a standard FV scheme can be used not only to enhance the accuracy of the resulting scheme, but also to hybridize
conservative and nonconservative numerical methods successfully. We refer the reader to
\cite{pidatella2019semi,chertock2021hybrid,abgrall2024new,abgrall2024staggered,GBCR}, where a clever use of primitive formulations of the
governing equations was made. On the other hand, we stress that nonconservative methods per se cannot be used for accurately solving
hyperbolic systems of conservation laws since nonconservative numerical schemes typically converge to non-entropy (non-physical) weak
solutions, as demonstrated in \cite{hou1994nonconservative,abgrall2010comment}.

A similar idea of obtaining additional degrees of freedom by evolving several pieces of information was used in the methods on overlapping
cells. These methods include both FV \cite{Liu2005,LSTZ2,LSX,XLDLS} and discontinuous Galerkin \cite{LSTZ1,XLDLS,XLS1} ones. In the FV
methods on overlapping cells, hierarchical reconstruction limiters are used to achieve high-order non-oscillatory approximations of the
computed solution, whose cell averages on overlapping cells are evolved in time.

In this paper, we develop a novel second-order semi-discrete dual formulation finite-volume (DF-FV) method.
To this end, we first consider the following nonconservative system, which is equivalent to \eref{1.1} for smooth solutions:
\begin{equation}
\bm V_t+\widetilde{\bm F}(\bm V)_x+\widetilde{\bm G}(\bm V)_y=B(\bm V)\bm V_x+C(\bm V)\bm V_y,
\label{1.2}
\end{equation}
where $\bm V\in\mathbb R^M$ is the vector of primitive variables, $\widetilde{\bm F},\widetilde{\bm G}:\mathbb R^M\to\mathbb R^M$, and
$B,C\in\mathbb R^{M\times M}$. To cite an example, one may consider the Euler equations of gas dynamics, which read as \eref{1.1} with
\begin{equation}
\bm U=(\rho,\rho u,\rho v,E)^\top,~~\bm F(\bm U)=(\rho u,\rho u^2+p,\rho uv,u(E+p))^\top,~~
\bm G(\bm U)=(\rho u,\rho uv,\rho v^2+p,v(E+p))^\top,
\label{1.3f}
\end{equation}
where $\rho$ is the density, $u$ and $v$ are the $x$- and $y$-velocities, $E$ is the total energy, and $p$ is the pressure. We use the
classical closure, obtained with the help of the equation of state of ideal fluids,
\begin{equation}
E=\frac{p}{\gamma-1}+\hf\rho(u^2+v^2),
\label{1.4f}
\end{equation}
in which $\gamma$ is the specific heat ratio. The system \eref{1.1}, \eref{1.3f}--\eref{1.4f} can be rewritten in the nonconservative form
in many different ways, for instance, using the primitive variables $\rho$, $u$, $v$, and $p$, for which the corresponding nonconservative
system reads as \eref{1.2} with
\begin{equation*}
\bm V=(\rho,u,v,p)^\top,\quad\widetilde{\bm F}(\bm V)=\Big(\rho u,\frac{u^2}{2},0,pu\Big)^\top,\quad
\widetilde{\bm G}(\bm V)=\Big(\rho v,0,\frac{v^2}{2},pv\Big)^\top,
\end{equation*}
\begin{equation*}
B(\bm V)=\begin{pmatrix}0&0&0&0\\0&0&0&-\frac{1}{\rho}\\0&0&-u&0\\0&-(\gamma-1)p&0&0\end{pmatrix},\quad
C(\bm V)=\begin{pmatrix}0&0&0&0\\0&-v&0&0\\0&0&0&-\frac{1}{\rho}\\0&0&-(\gamma-1)p&0\end{pmatrix}.
\end{equation*}

We discretize the systems \eref{1.1} and \eref{1.2} on overlapping meshes using a semi-discrete approach, and then solve the resulting
systems of ordinary differential equations (ODEs) simultaneously. The numerical fluxes for the conservative system \eref{1.1} are taken
simply as $\bm F(\bm U(\bm V))$ and $\bm G(\bm U(\bm V))$, while the discretization of the nonconservative system \eref{1.2} is more
involved as its solutions cannot be understood in the sense of distributions. In \cite{dal1995definition}, a concept of Borel measure
solutions of nonconservative systems was introduced (see also \cite{LeF02,LeF}) and later utilized to develop path-conservative numerical
methods; see, e.g., \cite{CLMP,CMP2017,diaz2019path,pares2006numerical,Par2009} and references therein. Here, we discretize \eref{1.2} using
a modified version of the Riemann-problem-solver-free path-conservative central-upwind (PCCU) scheme from \cite{diaz2019path}. The
modification is intended to reduce the amount of numerical dissipation present in the original PCCU scheme and is carried out by replacing
the central-upwind (CU) numerical flux from \cite{kurganov2001semidiscrete}, which was used in \cite{diaz2019path}, with a less dissipative
CU numerical flux from \cite{KLin}.

Since the conservative numerical fluxes do not employ any limiting procedures, one can expect the computed $\bm U$ to be oscillatory. At the
same time, variables $\bm V$, computed in a non-oscillatory manner by the PCCU scheme, may converge to a non-physical weak solution. We
therefore introduce a conservative post-processing, which couples the evolution of the two sets of variables. The resulting numerical
solution is (essentially) oscillation-free and the scheme converges to the physically relevant solution of \eref{1.1}.

We test the proposed DF-FV methods on several benchmarks for one- and two-dimensional Euler equations of gas dynamics. In these examples, we
also demonstrate that the 1-D DF-FV method outperforms the second-order central scheme on overlapping cells from \cite{Liu2005}. In
addition, we would like to point out that the proposed dual formulation framework can be applied to several contexts in which the primitive
formulation is preferable over the conservative one. Applications, left for upcoming works, include development of adaptive algorithms,
asymptotic-preserving schemes for the Euler equations of gas dynamics and thermal rotating shallow water equations in all Mach and Rossby
regimes, respectively, as well as robust hybrid methods for compressible multifluid flows.

The rest of the paper is organized as follows. In \S\ref{sec2}, we introduce the new one-dimensional (1-D) DF-FV method, and then extend it
to the 2-D case in \S\ref{sec3}. In \S\ref{sec4}, we report the results of several numerical experiments for both 1-D and 2-D Euler
equations of gas dynamics. Finally, in \S\ref{sec5}, we provide concluding remarks and discuss future perspectives.

\section{One-Dimensional Semi-Discrete DF-FV Method}\label{sec2}
In this section, we introduce the new semi-discrete DF-FV method for the 1-D version of \eref{1.1}:
\begin{equation}
\bm U_t+\bm F(\bm U)_x=\bm0,
\label{2.1}
\end{equation}
whose nonconservative formulation reads as
\begin{equation}
\bm V_t+\widetilde{\bm F}(\bm V)_x=B(\bm V)\bm V_x.
\label{2.2}
\end{equation}

We consider overlapping FV meshes consisting of uniform cells $I_j=[x_\jmh,x_\jph]$, $j=1,\ldots,N$ and $I_\jph=[x_j,x_{j+1}]$,
$j=0,\ldots,N$ with $x_{j+1}=x_\jph+\dx/2=x_j+\dx$ for all $j$. As in all FV methods, the computed quantities are cell averages of $\bm U$
and $\bm V$, which are obtained on the above two grids, namely,
$$
\xbar{\bm U}_j:\approx\frac{1}{\dx}\int\limits_{I_j}\bm U(x,t)\,{\rm d}x\quad\mbox{and}\quad
\xbar{\bm V}_\jph:\approx\frac{1}{\dx}\int\limits_{I_\jph}\bm V(x,t)\,{\rm d}x.
$$
Note that both $\xbar{\bm U}_j$ and $\xbar{\bm V}_\jph$, like many other indexed quantities below, depend on time, but we omit this
dependence for the sake of notation brevity.

The semi-discretization of the conservative system \eref{2.1} is obtained by integrating \eref{2.1} in space over the cells $I_j$, which
results in
\begin{equation}
\frac{\rm d}{{\rm d}t}\,\xbar{\bm U}_j=-\frac{1}{\dx}\Big[\bm{{\cal F}}_\jph-\bm{{\cal F}}_\jmh\Big],
\label{2.3}
\end{equation}
in which we take the following simple numerical fluxes:
\begin{equation}
\bm{{\cal F}}_\jph=\bm F\big(\bm U_\jph\big),\quad\bm U_\jph:=\bm U\big(\,\xbar{\bm V}_\jph\big),
\label{2.4}
\end{equation}
where $\bm U(\bm V)$ denotes the transformation from primitive to conserved variables. Note that for second-order methods, cell averages and
point values formally differ by ${\cal O}((\dx)^2)$, which makes the transformation used in \eref{2.4} straightforward, while higher order
extensions would require a suitable higher order reconstruction of point values.

The semi-discretization of the nonconservative system \eref{2.2} is obtained through the modified PCCU scheme and reads
\begin{equation}
\frac{\rm d}{{\rm d}t}\Vbar_\jph=-\frac{1}{\dx}\bigg[\widetilde{\bm{{\cal F}}}_{j+1}-\widetilde{\bm{{\cal F}}}_j-\bm B_\jph
-\frac{a^+_j}{a^+_j-a^-_j}\bm B_{\bm\Psi,j}+\frac{a^-_{j+1}}{a^+_{j+1}-a^-_{j+1}}\bm B_{\bm\Psi,j+1}\bigg],
\label{2.5}
\end{equation}
where $\widetilde{\bm{{\cal F}}}_j$ are the CU numerical fluxes from \cite{KLin} given by
\begin{equation}
\widetilde{\bm{{\cal F}}}_j=\frac{a^+_j\widetilde{\bm F}\big(\bm V^-_j\big)-a^-_j\widetilde{\bm F}\big(\bm V^+_j\big)}{a^+_j-a^-_j}+
\frac{a^+_ja^-_j}{a^+_j-a^-_j}\Big(\bm V^+_j-\bm V^-_j-\delta\bm V_j\Big).
\label{2.6}
\end{equation}
Here, $\bm V^\pm_j$ are one-sided point values of $\bm V$ obtained using a piecewise linear reconstruction applied to the local
characteristic variables $\bm\Gamma$ of \eref{1.2}. To this end, we follow \cite{micalizzitoro2024,micalizzi2025algorithms} and introduce
$$
\bm\Gamma_\jmh=Q_\jph^{-1}\xbar{\bm V}_\jmh,\quad\bm\Gamma_\jph=Q_\jph^{-1}\xbar{\bm V}_\jph,\quad
\bm\Gamma_{j+\frac{3}{2}}=Q_\jph^{-1}\xbar{\bm V}_{j+\frac{3}{2}},
$$
where $Q_\jph$ is a matrix such that $Q_\jph^{-1}{\cal A}_\jph Q_\jph$ is a diagonal matrix and
${\cal A}_\jph={\cal A}\big(\xbar{\bm V}_\jph\big):=\frac{\partial\widetilde{\bm F}}{\partial\bm V}\big(\xbar{\bm V}_\jph\big)-
B\big(\xbar{\bm V}_\jph\big)$. Equipped with $\bm\Gamma_\jmh$, $\bm\Gamma_\jph$, and $\bm\Gamma_{j+\frac{3}{2}}$, we apply a generalized
minmod limiter (see, e.g., \cite{LN,sweby1984high}) to evaluate
$$
(\bm\Gamma_x)_\jph={\rm minmod}\left(\theta\,\frac{\bm\Gamma_\jph-\bm\Gamma_\jmh}{\dx},
\frac{\bm\Gamma_{j+\frac{3}{2}}-\bm\Gamma_\jmh}{2\dx},\,\theta\,\frac{\bm\Gamma_{j+\frac{3}{2}}-\bm\Gamma_\jph}{\dx}\right),\quad
\theta\in[1,2],
$$
where the minmod function, defined as
\begin{equation*}
{\rm minmod}(c_1,c_2,\ldots)=
\left\{\begin{aligned}
&\min_ic_i&&\mbox{if}~c_i>0~~\forall i,\\
&\max_ic_i&&\mbox{if}~c_i<0~~\forall i,\\
&0&&\mbox{otherwise},
\end{aligned}\right.
\end{equation*}
is applied in a component-wise manner. We then obtain
$$
\bm\Gamma_j^+=\bm\Gamma_\jph-\frac{\dx}{2}(\bm\Gamma_x)_\jph,\quad\bm\Gamma_{j+1}^-=\bm\Gamma_\jph+\frac{\dx}{2}(\bm\Gamma_x)_\jph,
$$
and hence the corresponding point values of $\bm V$ are
$$
\bm V^+_j=Q_\jph\bm\Gamma_j^+,\quad\bm V^-_{j+1}=Q_\jph\bm\Gamma_{j+1}^-.
$$
In \eref{2.5} and \eref{2.6}, $a^\pm_j$ are the one-sided local speeds of propagation, which can be estimated by
\begin{equation*}
a^-_j=\min\left\{\lambda_1(\bm V^-_j),\lambda_1(\bm V^+_j),0\right\},\quad
a^+_j=\max\left\{\lambda_M(\bm V^-_j),\lambda_M(\bm V^+_j),0\right\}.
\end{equation*}
where $\lambda_1(\bm V)\le\ldots\le\lambda_M(\bm V)$ are the eigenvalues of ${\cal A}(\bm V)$. The term $\delta\bm V_j$ in \eref{2.6}
represents a ``built-in'' anti-diffusion and is given by (see \cite{KLin}),
\begin{equation*}
\delta\bm V_j={\rm minmod}\left(\bm V_j^*-\bm V^-_j,\bm V^+_j-\bm V_j^*\right),\quad
\bm V_j^*:=\frac{a^+_j\bm V^+_j-a^-_j\bm V^-_j-\widetilde{\bm F}(\bm V^+_j)+\widetilde{\bm F}(\bm V^-_j)}{a^+_j-a^-_j}.
\end{equation*}
Finally, the terms $\bm B_\jph$ and $\bm B_{\bm\Psi,j}$ in \eref{2.5} are given by
\begin{equation*}
\bm B_\jph=B\big(\,\xbar{\bm V}_\jph\big)\big(\bm V_{j+1}^--\bm V_j^+\big),\quad
\bm B_{\bm\Psi,j}=\hf\left[B(\bm V_j^-)+B\big(\bm V_j^+\big)\right]\big(\bm V_j^+-\bm V_j^-\big),
\end{equation*}
where the former term is obtained by applying a second-order quadrature to $\int_{I_j}B(\bm V)\bm V\,{\rm d}x$ and the latter one is
derived using a linear path connecting $\bm V_j^-$ with $\bm V_j^+$; see \cite{diaz2019path} for details.
\begin{remark}
It should be observed that the numerical flux \eref{2.6} is different from the one used in the original version of the PCCU scheme
introduced in \cite{diaz2019path}. The difference is attributed to the presence of the anti-diffusion term $\delta\bm V_j$ that helps to
reduce the numerical dissipation present in the PCCU scheme and thus to enhance the resolution of contact waves as it was demonstrated in
\cite{KLin}, where the conservative formulation of the compressible Euler equations was considered.
\end{remark}
\begin{remark}
We would like to stress that the proposed DF-FV method is not tied to the PCCU scheme and, in principle, one can numerically solve the
nonconservative system \eref{2.2} using an alternative second-order stable numerical method instead. However, the PCCU scheme seems to be a
reasonable choice, thanks to its distinctive feature: once the path has been selected, the resulting method is not sensitive to a particular
choice on the nonconservative formulation; see \cite{diaz2019path}.
\end{remark}

The ODE systems \eref{2.3} and \eref{2.5} should be numerically solved by a stable and sufficiently accurate ODE solver. However, the
solution obtained upon completion of a time step evolution will have two significant drawbacks. First, the solution realized by
$\big\{\,\xbar{\bm U}_j(t+\dt)\big\}$ will likely be oscillatory since no limiting procedures are employed in the computation of numerical
fluxes in \eref{2.4}. Second, the solution realized by $\big\{\xbar{\bm V}_\jph(t+\dt)\big\}$ will not necessarily be conservative, that 
is, $\sum_j\bm U\big(\xbar{\bm V}_\jph(t+\dt)\big)$ may not be equal to $\sum_j\bm U\big(\xbar{\bm V}_\jph(t)\big)$. Therefore, to
ensure that the resulting numerical solution is (essentially) oscillation-free and converges to the physically relevant solution of
\eref{2.1}, we propose the conservative post-processing procedure presented in the next subsection.

\subsection{Post-Processing}\label{sec21}
Let us assume that the solution was evolved from time $t$ to $t+\dt$ with the help of an ODE solver, and denote the obtained solutions by
$\big\{\,\xbar{\bm U}_j^{\,*}\big\}$ and $\big\{\xbar{\bm V}_\jph^{\,*}\big\}$. Our goal is to modify these values through a suitable
post-processing procedure to obtain non-oscillatory sets of $\big\{\,\xbar{\bm U}_j(t+\dt)\big\}$ and
$\big\{\xbar{\bm V}_\jph(t+\dt)\big\}$.

The proposed post-processing can be presented algorithmically through the following four steps.

\vskip3pt
\noindent
{\bf Step 1.} Compute the conserved variables at $x=x_\jph$ using the transformation from the primitive ones:
$$
\bm U_\jph^*:=\bm U\big(\xbar{\bm V}_\jph^{\,*}\big).
$$

\vskip3pt
\noindent
{\bf Step 2.} Perform the piecewise linear reconstruction for $\bm U$ using the slopes computed by the minmod limiter
$$
(\bm U_x)_j^*=2\,{\rm minmod}\left(\frac{\,\xbar{\bm U}_j^{\,*}-\bm U_\jmh^*}{\dx},\frac{\bm U_\jph^*-\,\xbar{\bm U}_j^{\,*}}{\dx}\right),
$$
which results in
\begin{equation}
\bm U_\jph^{*,-}:=\,\xbar{\bm U}_j^{\,*}+\frac{\dx}{2}(\bm U_x)_j^*,\quad
\bm U_\jph^{*,+}:=\,\xbar{\bm U}_{j+1}^{\,*}-\frac{\dx}{2}(\bm U_x)_{j+1}^*.
\label{2.7}
\end{equation}

\vskip3pt
\noindent
{\bf Step 3.} Set
\begin{equation}
\bm U_\jph^{**}:=\hf\Big(\bm U_\jph^{*,-}+\bm U_\jph^{*,+}\Big),
\label{2.8}
\end{equation}
and recompute the primitive variables at $x=x_\jph$ using the transformation from the conservative ones:
$$
\xbar{\bm V}_\jph(t+\dt)=\bm V\big(\bm U_\jph^{**}\big).
$$

\vskip3pt
\noindent
{\bf Step 4.} Correct the conserved variables by setting
\begin{equation}
\xbar{\bm U}_j(t+\dt)=\hf\Big(\bm U_\jmh^{**}+\bm U_\jph^{**}\Big).
\label{2.9}
\end{equation}

\vskip7pt
It is essential to emphasize that the post-processing is conservative, as demonstrated in the following proposition.
\begin{proposition}
The post-processed cell averages of the conserved variables satisfy
\begin{equation*}
\sum_j\,\xbar{\bm U}_j(t+\dt)=\sum_j\,\xbar{\bm U}_j^{\,*}.
\end{equation*}
\end{proposition}
\begin{proof}
The proof consists of simple direct computations (assuming no contributions from the boundary terms):
$$
\begin{aligned}
\sum_j\,\xbar{\bm U}_j(t+\dt)&\stackrel{\eref{2.9}}{=}\hf\sum_j\bm U_\jmh^{**}+\hf\sum_j\bm U_\jph^{**}=\sum_j\bm U_\jph^{**}
\stackrel{\eref{2.8}}{=}\hf\sum_j\bm U_\jph^{*,-}+\hf\sum_j\bm U_\jph^{*,+}\\
&\stackrel{\eref{2.7}}{=}\hf\sum_j\Big[\,\xbar{\bm U}_j^{\,*}+\frac{\dx}{2}(\bm U_x)_j^*\Big]+
\hf\sum_j\Big[\,\xbar{\bm U}_{j+1}^{\,*}-\frac{\dx}{2}(\bm U_x)_{j+1}^*\Big]=\sum_j\,\xbar{\bm U}_j^{\,*}.
\end{aligned}
$$
\end{proof}

\subsection{Linear Stability Analysis}
It is instructive to study a linear $L^\infty$ stability of the developed DF-FV method. To this end, we apply the scheme to the linear 
advection equation
\begin{equation}
U_t+aU_x=0,
\label{2.10}
\end{equation}
where $a$ is a positive constant. In this case, the conservative and primitive formulations coincide, and nonconservative products are not
present.

For the sake of simplicity, we apply the first-order forward Euler time discretization, for which the fully discrete DF-FV scheme can be
written as follows. Starting from the discrete solution ${\xbar U_j(t), \xbar U_\jph(t)}$, we first reconstruct the point values 
${U_j^\pm(t)}$ from ${\xbar U_\jph(t)}$ using a generalized minmod limiter. We then evolve both $\xbar U_j$ and $\xbar U_\jph$ in time to 
obtain
\begin{align}
&\xbar U^{\,*}_\jph=\,\xbar U_\jph(t)-\frac{a\dt}{\dx}\left[U^-_{j+1}(t)-U^-_j(t)\right],\label{2.11}\\
&\xbar U^{\,*}_j=\,\xbar U_j(t)-\frac{a\dt}{\dx}\left[U_\jph(t)-U_\jmh(t)\right],\label{2.12}
\end{align}
where \eref{2.11} is, in fact, a second-order (in space) upwind scheme, which the PCCU scheme reduces to when applied to the linear
advection \eref{2.10} with $a>0$. Next, we apply the post-processing from \S\ref{sec21}. Namely, we compute the reconstructed interface 
values $U^{*,\pm}_\jph$ using \eref{2.7} and evaluate
\begin{equation}
\xbar U_\jph(t+\dt)=\hf\left[U^{*,-}_\jph+U^{*,+}_\jph\right]
\label{2.13}
\end{equation}
and
\begin{equation}
\xbar U_j(t+\dt)=\hf\left[\,\xbar U_\jmh(t+\dt)+\,\xbar U_\jph(t+\dt)\right].
\label{2.14}
\end{equation}

We now assume the bounds $\|\,\xbar U_\jph(t)\|_\infty\le K$ and $\|\,\xbar U_j(t)\|_\infty\le K$ and prove that the updated values 
$\|\,\xbar U_\jph(t+\dt)\|_\infty$ and $\|\,\xbar U_j(t+\dt)\|_\infty$ remain bounded by $K$. First, we observe that \eref{2.11} can be 
rewritten as
$$
\xbar U^{\,*}_\jph=\frac{U^+_j(t)+U^-_{j+1}(t)}{2}-\frac{a\dt}{\dx}\left[U^-_{j+1}(t)-U^-_j(t)\right]
=\left(\hf-\frac{a\dt}{\dx}\right)U^-_{j+1}(t)+\hf U^+_j(t)+\frac{a\dt}{\dx}U^-_j(t).
$$
Imposing the CFL condition
\begin{equation}
\frac{a\dt}{\dx}\le\hf
\label{2.15}
\end{equation}
and using the non-oscillatory nature of the generalized minmod reconstruction that ensures $|U^\pm_j|\le K$, we obtain
\begin{equation}
\big|\,\xbar U^{\,*}_\jph\big|\le\left(\hf-\frac{a\dt}{\dx}\right)|U^-_{j+1}(t)|+\hf|U^+_j(t)|+\frac{a\dt}{\dx}|U^-_j(t)|
\le\left(\hf-\frac{a\dt}{\dx}+\hf+\frac{a\dt}{\dx}\right)K=K.
\label{2.16}
\end{equation}
Next, \eref{2.12} together with \eref{2.14}, written at the previous time level $t$ rather than at $t+\dt$, yields
$$
\xbar U^{\,*}_j=\frac{\,\xbar U_\jph(t)+\,\xbar U_\jmh(t)}{2}-\frac{a\dt}{\dx}\left[\,\xbar U_\jph(t)-\,\xbar U_\jmh(t)\right]
=\left(\hf-\frac{a\dt}{\dx}\right)\,\xbar U_\jph(t)+\left(\hf+\frac{a\dt}{\dx}\right)\,\xbar U_\jmh(t),
$$
which, under the same CFL restriction \eref{2.15}, leads to
\begin{equation}
\big|\,\xbar U^{\,*}_j\big|\le\left(\hf-\frac{a\dt}{\dx}\right)\big|\,\xbar U_\jph(t)\big|+
\left(\hf+\frac{a\dt}{\dx}\right)\big|\,\xbar U_\jmh(t)\big|\le\left(\hf-\frac{a\dt}{\dx}+\hf+\frac{a\dt}{\dx}\right)K=K.
\label{2.17}
\end{equation}
Finally, combining equations \eref{2.13}–\eref{2.14} with the bounds \eref{2.16}–\eref{2.17}, and noting once again that the generalized 
minmod reconstruction is non-oscillatory and satisfies $\big|U^{*,\pm}_\jph\big|\le K$, we conclude that both 
$\xbar U_\jph(t+\dt)$ and $\xbar U_j(t+\dt)$ remain bounded by $K$. This establishes the $L^\infty$ stability of the DF–FV method for 
linear problems \eqref{2.10} under the CFL condition \eref{2.15}.

\section{Two-Dimensional Semi-Discrete DF-FV Method}\label{sec3}
We now extend the proposed semi-discrete DF-FV method to the 2-D case, in which the conservative and nonconservative (primitive)
formulations of the governing equations are given by \eref{1.1} and \eref{1.2}, respectively. 

We consider overlapping Cartesian meshes (see the sketch in Figure \ref{fig31}) consisting of uniform cells
$I_{j,k}:=[x_\jmh,x_\jph]\times[y_\kmh,y_\kph]$, $j=1,\ldots,N_x$, $k=1,\ldots,N_y$, $I_{\jph,k}^x=[x_j,x_{j+1}]\times[y_\kmh,y_\kph]$,
$j=0,\ldots,N_x$, $k=1,\ldots,N_y$, and $I_{j,\kph}^y=[x_\jmh,x_\jph]\times[y_k,y_{k+1}]$, $j=1,\ldots,N_x$, $k=0,\ldots,N_y$ with
$x_{j+1}=x_\jph+\dx/2=x_j+\dx$ for all $j$ and $y_{k+1}=y_\kph+\dy/2=y_k+\dy$ for all $k$. The computed quantities are the following cell
averages obtained on the above three grids:
$$
\xbar{\bm U}_{j,k}:\approx\frac{1}{\dx\dy}\int\limits_{I_{j,k}}\bm U(x,t)\,{\rm d}x{\rm d}y,~~
\xbar{\bm V}^{\,x}_{\jph,k}:\approx\frac{1}{\dx\dy}\int\limits_{I_{\jph,k}^x}\bm V(x,t)\,{\rm d}x{\rm d}y,~~
\xbar{\bm V}^{\,y}_{j,\kph}:\approx\frac{1}{\dx\dy}\int\limits_{I_{j,\kph}^y}\bm V(x,t)\,{\rm d}x{\rm d}y.
$$
\begin{figure}[ht!]
\vskip-40pt
\centering
\scalebox{0.9}{
\begin{tikzpicture}
\def\s{2}
\begin{scope}[shift={(-0.5,0)}]  
\draw[->] (0,0) -- (7,0) node[below] {$x$};
\draw[->] (0,0) -- (0,6) node[left] {$y$};
\foreach\x/\lab in {1.5/{$x_{j-1}$},2.5/{$x_\jmh$},3.5/{$x_j$},4.5/{$x_\jph$},5.5/{$x_{j+1}$}}{\draw (\x,0.05) -- (\x,-0.05);
\node[below] at (\x,-0.05){\lab};}
\foreach \y/\lab in {1/{$y_{k-1}$},2/{$y_\kmh$},3/{$y_k$},4/{$y_\kph$},5/{$y_{k+1}$}}{\draw (0.05,\y) -- (-0.05,\y);
\node[left] at (-0.05,\y){\lab};}
\end{scope}
\draw[ultra thick] (\s,\s) rectangle (2*\s,2*\s);
\fill[pattern={Lines[angle=45, distance=4pt, line width=0.4pt]}, pattern color=red] (0.5*\s,\s) rectangle (1.5*\s,2*\s);
\draw[red] (0.5*\s,\s) rectangle (1.5*\s,2*\s);
\fill[pattern={Lines[angle=45, distance=4pt, line width=0.4pt]}, pattern color=red] (1.5*\s,\s) rectangle (2.5*\s,2*\s);
\draw[red] (1.5*\s,\s) rectangle (2.5*\s,2*\s);
\fill[pattern={Lines[angle=135, distance=4pt, line width=0.4pt]}, pattern color=blue] (\s,0.5*\s) rectangle (2*\s,1.5*\s);
\draw[blue] (\s,0.5*\s) rectangle (2*\s,1.5*\s);
\draw[blue] (\s,0.5*\s) rectangle (2*\s,1.5*\s);
\fill[pattern={Lines[angle=135, distance=4pt, line width=0.4pt]}, pattern color=blue] (\s,1.5*\s) rectangle (2*\s,2.5*\s);
\draw[blue] (\s,1.5*\s) rectangle (2*\s,2.5*\s);
\end{tikzpicture}}
\caption{\sf Sketch of the overlapping cells: $I_{j,k}$ (black-bordered); $I_{\jmh,k}^x$ and $I_{\jph,k}^x$ (red-filled); $I_{j,\kmh}^y$
and $I_{j,\kph}^y$ (blue-filled).
\label{fig31}}
\end{figure}

As in the 1-D case, the semi-discretization of the conservative system \eref{1.1} is obtained by integrating \eref{1.1} in space over the
cells $I_{j,k}$, which leads to
\begin{equation*}
\frac{\rm d}{{\rm d}t}\,\xbar{\bm U}_{j,k}=-\frac{1}{\dx}\Big[\bm{{\cal F}}_{\jph,k}-\bm{{\cal F}}_{\jmh,k}\Big]-
\frac{1}{\dy}\Big[\bm{{\cal G}}_{j,\kph}-\bm{{\cal G}}_{j,\kmh}\Big],
\end{equation*}
with the following second-order numerical fluxes:
$$
\bm{{\cal F}}_{\jph,k}=\bm F\big(\bm U_{\jph,k}\big),~~\bm U_{\jph,k}:=\bm U\big(\,\xbar{\bm V}^{\,x}_{\jph,k}\big)\quad\mbox{and}\quad
\bm{{\cal G}}_{j,\kph}=\bm G\big(\bm U_{j,\kph}\big),~~\bm U_{j,\kph}:=\bm U\big(\,\xbar{\bm V}^{\,y}_{j,\kph}\big).
$$

Two {\em independent} semi-discretizations of the nonconservative system \eref{1.2} are obtained through the modified 2-D PCCU scheme
applied on the $I_{\jph,k}^x$ and $I_{j,\kph}^y$ meshes, respectively, as follows:
\allowdisplaybreaks
\begin{align}
&\begin{aligned}
\frac{\rm d}{{\rm d}t}\,\xbar{\bm V}^{\,x}_{\jph,k}=&-\frac{1}{\dx}\bigg[\widetilde{\bm{{\cal F}}}^{\,x}_{j+1,k}-
\widetilde{\bm{{\cal F}}}^{\,x}_{j,k}-\bm B^x_{\jph,k}-\frac{a^{x,+}_{j,k}}{a^{x,+}_{j,k}-a^{x,-}_{j,k}}\,\bm B^x_{\bm\Psi,j,k}+
\frac{a^{x,-}_{j+1,k}}{a^{x,+}_{j+1,k}-a^{x,-}_{j+1,k}}\,\bm B^x_{\bm\Psi,j+1,k}\bigg]\\
&-\frac{1}{\dy}\bigg[\widetilde{\bm{{\cal G}}}^{\,x}_{\jph,\kph}-\widetilde{\bm{{\cal G}}}^{\,x}_{\jph,\kmh}-\bm C^x_{\jph,k}\\
&\hspace{1.6cm}-\frac{b^{x,+}_{\jph,\kmh}}{b^{x,+}_{\jph,\kmh}-b^{x,-}_{\jph,\kmh}}\,\bm C^x_{\bm\Psi,\jph,\kmh}+
\frac{b^{x,-}_{\jph,\kph}}{b^{x,+}_{\jph,\kph}-b^{x,-}_{\jph,\kph}}\,\bm C^x_{\bm\Psi,\jph,\kph}\bigg]
\end{aligned}
\label{3.1}\\
&\begin{aligned}
\frac{\rm d}{{\rm d}t}\,\xbar{\bm V}^{\,y}_{j,\kph}=&-\frac{1}{\dx}\bigg[\widetilde{\bm{{\cal F}}}^{\,y}_{\jph,\kph}-
\widetilde{\bm{{\cal F}}}^{\,y}_{\jmh,\kph}-\bm B^y_{j,\kph}\\
&\hspace{1.25cm}-\frac{a^{y,+}_{\jmh,\kph}}{a^{y,+}_{\jmh,\kph}-a^{y,-}_{\jmh,\kph}}\,\bm B_{\bm\Psi,\jmh,\kph}^y+
\frac{a^{y,-}_{\jph,\kph}}{a^{y,+}_{\jph,\kph}-a^{y,-}_{\jph,\kph}}\,\bm B^y_{\bm\Psi,\jph,\kph}\bigg]\\
&-\frac{1}{\dy}\bigg[\widetilde{\bm{{\cal G}}}^{\,y}_{j,k+1}-\widetilde{\bm{{\cal G}}}^{\,y}_{j,k}-\bm C^y_{j,\kph}
-\frac{b^{y,+}_{j,k}}{b^{y,+}_{j,k}-b^{y,-}_{j,k}}\,\bm C^y_{\bm\Psi,j,k}+
\frac{b^{y,-}_{j,k+1}}{b^{y,+}_{j,k+1}-b^{y,-}_{j,k+1}}\,\bm C^y_{\bm\Psi,j,k+1}\bigg]
\end{aligned}
\label{3.2}
\end{align}
Once again, we stress that the right-hand sides (RHSs) of \eref{3.1} and \eref{3.2} are computed independently and thus their computations
can be performed in parallel.

The terms on the RHS of \eref{3.1} are computed in a dimension-by-dimension manner and are given by
\allowdisplaybreaks
\begin{align*}
&\widetilde{\bm{{\cal F}}}^{\,x}_{j,k}=\frac{a^{x,+}_{j,k}\widetilde{\bm F}\big(\bm V^{x,-}_{j,k}\big)-
a^{x,-}_{j,k}\widetilde{\bm F}\big(\bm V^{x,+}_{j,k}\big)}{a^{x,+}_{j,k}-a^{x,-}_{j,k}}+
\frac{a^{x,+}_{j,k}a^{x,-}_{j,k}}{a^{x,+}_{j,k}-a^{x,-}_{j,k}a\,}\Big(\bm V^{x,+}_{j,k}-\bm V^{x,-}_{j,k}-\delta\bm V^x_{j,k}\Big),\\
&\delta\bm V^x_{j,k}={\rm minmod}\left(\bm V_{j,k}^{x,*}-\bm V^{x,-}_{j,k},\bm V^{x,+}_{j,k}-\bm V_{j,k}^{x,*}\right),\quad
\bm V_{j,k}^{x,*}=\frac{a^{x,+}_{j,k}\bm V^{x,+}_{j,k}-a^{x,-}_{j,k}\bm V^{x,-}_{j,k}-\widetilde{\bm F}(\bm V^{x,+}_{j,k})+
\widetilde{\bm F}(\bm V^{x,-}_{j,k})}{a^{x,+}_{j,k}-a^{x,-}_{j,k}},\\
&\widetilde{\bm{{\cal G}}}^x_{\jph,\kph}=\frac{b^{x,+}_{\jph,\kph}\widetilde{\bm G}\big(\bm V^{x,-}_{\jph,\kph}\big)-
b^{x,-}_{\jph,\kph}\widetilde{\bm G}\big(\bm V^{x,+}_{\jph,\kph}\big)}{b^{x,+}_{\jph,\kph}-b^{x,-}_{\jph,\kph}}\\
&\hspace{1.4cm}+\frac{b^{x,+}_{\jph,\kph}b^{x,-}_{\jph,\kph}}{b^{x,+}_{\jph,\kph}-b^{x,-}_{\jph,\kph}}\,
\Big(\bm V^{x,+}_{\jph,\kph}-\bm V^{x,-}_{\jph,\kph}-\delta\bm V^x_{\jph,\kph}\Big),\\
&\delta\bm V^x_{\jph,\kph}={\rm minmod}\Big(\bm V_{\jph,\kph}^{x,*}-\bm V^{x,-}_{\jph,\kph},
\bm V^{x,+}_{\jph,\kph}-\bm V_{\jph,\kph}^{x,*}\Big),\\
&\bm V_{\jph,\kph}^{x,*}=\frac{b^{x,+}_{\jph,\kph}\bm V^{x,+}_{\jph,\kph}-b^{x,-}_{\jph,\kph}\bm V^{x,-}_{\jph,\kph}-
\widetilde{\bm G}\big(\bm V^{x,+}_{\jph,\kph}\big)+\widetilde{\bm G}\big(\bm V^{x,-}_{\jph,\kph}\big)}
{b^{x,+}_{\jph,\kph}-b^{x,-}_{\jph,\kph}},\\
&\bm B^x_{\jph,k}=B\big(\xbar{\bm V}^x_{\jph,k}\big)\big(\bm V_{j+1,k}^{x,-}-\bm V_{j,k}^{x,+}\big),\\
&\bm B^x_{\bm\Psi,j,k}=\hf\Big[B(\bm V_{j,k}^{x,-})+B\big(\bm V_{j,k}^{x,+}\big)\Big]
\big(\bm V_{j,k}^{x,+}-\bm V_{j,k}^{x,-}\big),\\
&C^x_{\jph,k}=C\big(\xbar{\bm V}^x_{\jph,k}\big)\big(\bm V_{\jph,\kph}^{x,-}-\bm V_{\jph,\kmh}^{x,+}\big),\\
&\bm C^x_{\bm\Psi,\jph,\kph}=\hf\Big[C(\bm V_{\jph,\kph}^{x,-})+C\big(\bm V_{\jph,\kph}^{x,+}\big)\Big]
\big(\bm V_{\jph,\kph}^{x,+}-\bm V_{\jph,\kph}^{x,-}\big),
\end{align*}
where $\bm V^{x,\pm}_{j,k}$ and $\bm V^{x,\pm}_{\jph,\kph}$ are reconstructed point values obtained from the cell averages
$\big\{\xbar{\bm V}^{\,x}_{j,k}\big\}$. As in the 1-D case, we compute these values by exploiting the local characteristic decomposition. 
In the $x$-direction, we have
$$
\bm V^{x,+}_{j,k}=Q^x_{\jph,k}\bm\Gamma_{j,k}^{x,+},\quad\bm V^{x,-}_{j+1,k}=Q^x_{\jph,k}\bm\Gamma_{j+1,k}^{x,-},
$$
where
$$
\bm\Gamma_{j,k}^{x,+}=\bm\Gamma^x_{\jph,k}-\frac{\dx}{2}(\bm\Gamma_x)^x_{\jph,k},\quad
\bm\Gamma_{j+1,k}^{x,-}=\bm\Gamma^x_{\jph,k}+\frac{\dx}{2}(\bm\Gamma_x)^x_{\jph,k}
$$
with the slopes computed using the generalized minmod limiter:
\begin{equation}
(\bm\Gamma_x)^x_{\jph,k}={\rm minmod}\left(\theta\,\frac{\bm\Gamma^x_{\jph,k}-\bm\Gamma^x_{\jmh,k}}{\dx},
\frac{\bm\Gamma^x_{j+\frac{3}{2},k}-\bm\Gamma^x_{\jmh,k}}{2\dx},
\,\theta\,\frac{\bm\Gamma^x_{j+\frac{3}{2},k}-\bm\Gamma^x_{\jph,k}}{\dx}\right),\quad\theta\in[1,2].
\label{3.3}
\end{equation}
In \eref{3.3}, one has
$$
\bm\Gamma^x_{\jmh,k}=\big(Q^x_{\jph,k}\big)^{-1}\,\xbar{\bm V}^{\,x}_{\jmh,k},\quad
\bm\Gamma^x_{\jph,k}=\big(Q_{\jph,k}^x\big)^{-1}\,\xbar{\bm V}^{\,x}_{\jph,k},\quad
\bm\Gamma^x_{j+\frac{3}{2},k}=\big(Q_{\jph,k}^x\big)^{-1}\,\xbar{\bm V}^{\,x}_{j+\frac{3}{2},k},
$$
where the matrix $Q^x_{\jph,k}$ is such that $\big(Q^x_{\jph,k}\big)^{-1}{\cal A}^x_{\jph,k}Q^x_{\jph,k}$ is diagonal with
${\cal A}^x_{\jph,k}={\cal A}\big(\xbar{\bm V}^{\,x}_{\jph,k}\big):=
\frac{\partial\widetilde{\bm F}}{\partial\bm V}\big(\xbar{\bm V}^{\,x}_{\jph,k}\big)-B\big(\xbar{\bm V}^{\,x}_{\jph,k}\big)$. 

Similarly, in the $y$-direction, we have
$$
\bm V^{x,+}_{\jph,\kmh}=P^x_{\jph,k}\bm\Gamma_{\jph,\kmh}^{x,+},\quad\bm V^{x,-}_{\jph,\kph}=P^x_{\jph,k}\bm\Gamma_{\jph,\kph}^{x,-},
$$
where
$$
\bm\Gamma_{\jph,\kmh}^{x,+}=\bm\Gamma^x_{\jph,k}-\frac{\dy}{2}(\bm\Gamma_y)^x_{\jph,k},
\quad\bm\Gamma_{\jph,\kph}^{x,-}=\bm\Gamma_{\jph,k}^x+\frac{\dy}{2}(\bm\Gamma_y)^x_{\jph,k}
$$
with the slope given by
\begin{equation}
(\bm\Gamma_y)^x_{\jph,k}={\rm minmod}\left(\theta\,\frac{\bm\Gamma^x_{\jph,k}-\bm\Gamma^x_{\jph,k-1}}{\dy},
\frac{\bm\Gamma^x_{\jph,k+1}-\bm\Gamma^x_{\jph,k-1}}{2\dy},\,\theta\,\frac{\bm\Gamma^x_{\jph,k+1}-\bm\Gamma^x_{\jph,k}}{\dy}\right),\quad
\theta\in[1,2].
\label{3.4}
\end{equation}
In \eref{3.4}, one has
$$
\bm\Gamma^x_{\jph,k-1}=(P_{\jph,k}^x)^{-1}\xbar{\bm V}^{\,x}_{\jph,k-1},\quad
\bm\Gamma^x_{\jph,k}=(P^x_{\jph,k})^{-1}\xbar{\bm V}^{\,x}_{\jph,k},\quad
\bm\Gamma^x_{\jph,k+1}=(P_{\jph,k}^x)^{-1}\xbar{\bm V}^{\,x}_{\jph,k+1},
$$
where the matrix $P^x_{\jph,k}$ is such that $(P^x_{\jph,k})^{-1}{\cal B}^x_{\jph,k} P^x_{\jph,k}$ is diagonal with
${\cal B}^x_{\jph,k}={\cal B}\big(\xbar{\bm V}_{\jph,k}^{\,x}\big):=
\frac{\partial\widetilde{\bm G}}{\partial\bm V}\big(\xbar{\bm V}^{\,x}_{\jph,k}\big)-C\big(\xbar{\bm V}^{\,x}_{\jph,k}\big)$. 

The one-sided local speeds of propagation in \eref{3.1} are estimated by
\begin{equation*}
\begin{aligned}
&a^{x,-}_{j,k}=\min\left\{\lambda_1(\bm V^{x,-}_{j,k}),\lambda_1(\bm V^{x,+}_{j,k}),0\right\},&&
a^{x,+}_{j,k}=\max\left\{\lambda_M(\bm V^{x,-}_{j,k}),\lambda_M(\bm V^{x,+}_{j,k}),0\right\},\\
&b^{x,-}_{\jph,\kph}=\min\left\{\mu_1(\bm V^{x,-}_{\jph,\kph}),\mu_1(\bm V^{x,+}_{\jph,\kph}),0\right\},&&
b^{x,+}_{\jph,\kph}=\max\left\{\mu_M(\bm V^{x,-}_{\jph,\kph}),\mu_M(\bm V^{x,+}_{\jph,\kph}),0\right\},
\end{aligned}
\end{equation*}
where $\lambda_1(\bm V)\le\ldots\le\lambda_M(\bm V)$ are the eigenvalues of ${\cal A}(\bm V)$ and
$\mu_1(\bm V)\le\ldots\le\mu_M(\bm V)$ are the eigenvalues of ${\cal B}(\bm V)$. 

The terms on the RHS of \eref{3.2}, related to the update of $\xbar{\bm V}^{\,y}_{j,\kph}$, can be analogously computed through the same
formulae used to evaluate the RHS of \eref{3.1}, but with the first index shifted by $-\hf$, the second index shifted by $+\hf$, and the
superscripts ``$x$'' replaced by ``$y$''.
\begin{remark}
We would like to emphasize that, as in the 1-D case, the proposed 2-D DF-FV method is not tied to the PCCU scheme and the nonconservative
system \eref{1.2} can be numerically solved by an alternative second-order stable numerical method.
\end{remark}

Finally, the post-processing is carried out in a ``dimension-by-dimension'' manner by alternating sweeps in the $x$-direction, in which we
modify the cell averages of $\xbar{\bm U}$ and $\xbar{\bm V}^{\,x}$, and in the $y$-direction, in which we modify the cell averages of
$\xbar{\bm U}$ and $\xbar{\bm V}^{\,y}$. There are two straightforward implementation options: first, to perform the sweep in the
$x$-direction and then the one in the $y$-direction, or the other way around. Unfortunately, both of these options may lead to asymmetries
with respect to the two Cartesian directions. To prevent this, we average the results given by these two alternatives. Namely, using the
same notation as in the 1-D case, we start from $\big\{\,\xbar{\bm U}_{j,k}^{\,*}\big\}$, $\big\{\xbar{\bm V}_{\jph,k}^{\,x,*}\big\}$, and
$\big\{\xbar{\bm V}_{j,\kph}^{\,y,*}\big\}$, denoting the solution values evolved from time $t$ to $t+\dt$ by an ODE solver, and we
independently perform two sub-post-processings:

\smallskip
\noindent
$\bullet$ One with the sweeps in the $x$- and then $y$-direction, resulting in $\big\{\,\xbar{\bm U}_{j,k}^{\,xy}\big\}$,
$\big\{\xbar{\bm V}_{\jph,k}^{\,x,xy}\big\}$, and $\big\{\xbar{\bm V}_{j,\kph}^{\,y,xy}\big\}$;

\smallskip
\noindent
$\bullet$ The other one with the sweeps in the $y$- and then $x$-direction, resulting in $\big\{\,\xbar{\bm U}_{j,k}^{\,yx}\big\}$,
$\big\{\xbar{\bm V}_{\jph,k}^{\,x,yx}\big\}$, and $\big\{\xbar{\bm V}_{j,\kph}^{\,y,yx}\big\}$.

\smallskip
\noindent
Upon the completion of these two sub-post-processing steps, we set
$$
\begin{aligned}
&\xbar{\bm U}_{j,k}(t+\dt)=\hf\Big(\,\xbar{\bm U}_{j,k}^{\,xy}+\xbar{\bm U}_{j,k}^{\,yx}\Big),\\
&\xbar{\bm V}^{\,x}_{\jph,k}(t+\dt)=\hf\Big(\,\xbar{\bm V}_{\jph,k}^{\,x,xy}+\xbar{\bm V}_{\jph,k}^{\,x,yx}\Big),\quad
\xbar{\bm V}^{\,y}_{j,\kph}(t+\dt)=\hf\Big(\,\xbar{\bm V}_{j,\kph}^{\,y,xy}+\xbar{\bm V}_{j,\kph}^{\,y,yx}\Big).
\end{aligned}
$$

The first of the proposed sub-post-processings can be presented algorithmically through the following eight steps, with Steps 1--4
corresponding to the sweep in the $x$-direction and Steps 5--8 corresponding to the sweep in the $y$-direction.

\vskip3pt
\noindent
{\bf Step 1.} Compute the conserved variables at $(x_\jph,y_k)$ using the transformation from the primitive ones:
$$
\bm U_{\jph,k}^*=\bm U\big(\xbar{\bm V}_{\jph,k}^{\,x,*}\big).
$$

\vskip3pt
\noindent
{\bf Step 2.} Perform the piecewise linear reconstruction for $\bm U$ in the $x$-direction using the slopes computed by the minmod limiter
$$
(\bm U_x)_{j,k}^*=2\,{\rm minmod}\left(\frac{\,\xbar{\bm U}_{j,k}^{\,*}-\bm U_{\jmh,k}^*}{\dx},
\frac{\bm U_{\jph,k}^*-\,\xbar{\bm U}_{j,k}^{\,*}}{\dx}\right),
$$
which results in
\begin{equation*}
\bm U_{\jph,k}^{*,-}:=\,\xbar{\bm U}_{j,k}^{\,*}+\frac{\dx}{2}(\bm U_x)_{j,k}^*,\quad
\bm U_{\jph,k}^{*,+}:=\,\xbar{\bm U}_{j+1,k}^{\,*}-\frac{\dx}{2}(\bm U_x)_{j+1,k}^*.
\end{equation*}

\vskip3pt
\noindent
{\bf Step 3.} Set
\begin{equation*}
\bm U_{\jph,k}^{**}:=\hf\Big(\bm U_{\jph,k}^{*,-}+\bm U_{\jph,k}^{*,+}\Big),
\end{equation*}
and recompute the primitive variables at $(x_{\jph},y_k)$ using the transformation from the conservative ones:
$$
\xbar{\bm V}_{\jph,k}^{\,x,xy}=\bm V\big(\bm U_{\jph,k}^{**}\big).
$$

\vskip3pt
\noindent
{\bf Step 4.} Correct the conserved variables by setting
\begin{equation*}
\widetilde{\bm U}_{j,k}:=\hf\Big(\bm U_{\jmh,k}^{**}+\bm U_{\jph,k}^{**}\Big).
\end{equation*}

\vskip3pt
\noindent
{\bf Step 5.} Compute the conserved variables at $(x_j,y_\kph)$ using the transformation from the primitive ones:
$$
\bm U_{j,\kph}^*=\bm U\big(\xbar{\bm V}_{j,\kph}^{\,y,*}\big).
$$

\vskip3pt
\noindent
{\bf Step 6.} Perform the piecewise linear reconstruction for $\bm U$ in the $y$-direction using the slopes computed by the minmod limiter
$$
(\bm U_y)_{j,k}^*=2\,{\rm minmod}\left(\frac{\,\widetilde{\bm U}_{j,k}-\bm U_{j,\kmh}^*}{\dy},
\frac{\bm U_{j,\kph}^*-\,\widetilde{\bm U}_{j,k}}{\dy}\right),
$$
which results in
\begin{equation*}
\bm U_{j,\kph}^{*,-}:=\,\widetilde{\bm U}_{j,k}+\frac{\dy}{2}(\bm U_y)_{j,k}^*,\quad
\bm U_{j,\kmh}^{*,+}:=\,\widetilde{\bm U}_{j,k+1}^{\,*}-\frac{\dy}{2}(\bm U_y)_{j,k+1}^*.
\end{equation*}

\vskip3pt
\noindent
{\bf Step 7.} Set
\begin{equation*}
\bm U_{j,\kph}^{**}:=\hf\Big(\bm U_{j,\kph}^{*,-}+\bm U_{j,\kph}^{*,+}\Big),
\end{equation*}
and recompute the primitive variables at $(x_j,y_\kph)$ using the transformation from the conservative ones:
$$
\xbar{\bm V}^{\,y,xy}_{j,\kph}=\bm V\big(\bm U_{j,\kph}^{**}\big).
$$

\vskip3pt
\noindent
{\bf Step 8.} Correct the conserved variables by setting
\begin{equation*}
\xbar{\bm U}_{j,k}^{\,xy}=\hf\Big(\bm U_{j,\kmh}^{**}+\bm U_{j,\kph}^{**}\Big).
\end{equation*}

\begin{remark}
We would like to stress that many operations within the proposed DF-FV method can be carried out in parallel. The updates of $\bm V^x$ and
$\bm V^y$ can be performed independently, while, the update of $\bm U$ requires $\bm V^x$ and $\bm V^y$, but no reconstructions and simple
numerical flux evaluations. Therefore, with a suitable parallelization of the operations, the computational cost may be significantly
reduced, especially since the post-processing is performed only once per time step.
\end{remark}

\section{Numerical Examples}\label{sec4}
In this section, we present the numerical results obtained for the Euler equations of gas dynamics. In the 1-D case, these equations can be
written in either the conservative form \eref{2.1} with
$$
\bm U=(\rho,\rho u,E)^\top,\quad\bm F(\bm U)=(\rho u,\rho u^2+p,u(E+p))^\top,
$$
and the equation of state $E=\frac{p}{\gamma-1}+\hf\rho u^2$, or the nonconservative form \eref{2.2} with
\begin{equation*}
\bm V=(\rho,u,p)^\top,\quad\widetilde{\bm F}(\bm V)=\Big(\rho u,\frac{u^2}{2},pu\Big)^\top,\quad
B(\bm V)=\begin{pmatrix}0&0&0\\0&0&-\frac{1}{\rho}\\0&-(\gamma-1)p&0\end{pmatrix}.
\end{equation*}

In all of the numerical examples below, the time discretization is performed using the three-stage third-order strong stability-preserving
Runge-Kutta (SSPRK3) method (see, e.g., \cite{GKS,gottlieb2001strong}) with the time-step selected adaptively with the CFL number set to be
0.475. We remark that one can replace the SSPRK3 ODE solver with one's favourite alternative. Our selection is motivated by the good
stability properties of the SSPRK3 method.

The minmod parameter has been set to $\theta=1.3$ in all examples except Example 5, where $\theta=1.1$ was used. In all of the examples, we
set $\gamma=1.4$.

In Examples 2--5, we also compare the performance of the proposed DF-FV method with the second-order central scheme on overlapping cells
(CSOC) from \cite{Liu2005}, implemented on the same mesh as the DF-FV method.

\paragraph{Example 1 (accuracy test for unsteady isentropic vortex).} In the first example taken from
\cite{Lagrange2D,shu1998essentially,micalizzi2023efficient}, we numerically verify the accuracy of the proposed DF-FV method.

We consider a smooth unsteady vortex on the computational domain $[-10,10]\times[-10,10]$, endowed with periodic boundary conditions. The
initial data are
$$
\begin{aligned}
&\rho(x,y,0)=\left(1-\frac{(\gamma-1)\kappa^2}{2\gamma}\right)^{\frac{1}{\gamma-1}},\quad p(x,y,0)=\rho^\gamma(x,y,0),\\
&u(x,y,0)=1-\kappa y,\quad v(x,y,0)=1+\kappa x,\quad\kappa=\frac{5}{2\pi}e^{\frac{1-x^2-y^2}{2}}.
\end{aligned}
$$
The exact solution is given by $\bm U(x,y,t)=\bm U(x-t,y-t,0)$, modulo the periodic boundary conditions.

We compute the solution until the final time $t=0.1$ on a sequence of uniform $N\times N$ meshes with $N=100$, $200$, $400$, $800$, and
$1600$. The obtained results are reported in Table \ref{tab41}, in which we demonstrate that the expected order of convergence is achieved
in the $\bm U$-solution (we show the $\rho$-, $\rho u$-, and $E$-components), $\bm V^x$-solution (we show the $v$-component), and
$\bm V^y$-solution (we show the $p$-component). The same order is observed in all other components, not shown here for the sake of brevity.
\begin{table}[ht!]
\centering
\begin{tabular}{|c|cc|cc|cc|cc|cc|}
\hline 
$N$ & $\rho$-error & rate & $\rho u$-error & rate & $E$-error & rate & $v$-error & rate & $p$-error & rate \\
\hline
100 &1.08e-2&--  &2.59e-2&--  &6.36e-2&--  &6.15e-2&--  &3.33e-2&--  \\
200 &3.04e-3&1.82&6.44e-3&2.01&1.66e-2&1.94&1.58e-2&1.96&7.91e-3&2.07\\
400 &9.08e-4&1.74&1.68e-3&1.94&4.38e-3&1.93&4.12e-3&1.94&1.90e-3&2.06\\
800 &2.44e-4&1.90&4.25e-4&1.98&1.08e-3&2.01&1.04e-3&1.99&4.53e-4&2.07\\
1600&5.86e-5&2.06&1.01e-4&2.07&2.44e-4&2.15&2.56e-4&2.02&1.06e-4&2.10\\
\hline
\end{tabular}
\caption{\sf Example 1: $L^1$-errors and corresponding convergence rates for the $\rho$-, $\rho u$-, and $E$-components of the
$\bm U$-solution, $v$-component of the $\bm V^x$-solution, and $p$-component of the $\bm V^y$-solution.\label{tab41}}
\end{table}

It should be observed that in this example, the solution is smooth and thus it can, in principle, be computed without post-processing.
We have verified that, in this case, the second-order convergence is also achieved. These results are, however, omitted since the presented
DF-FV method without the post-processing is impractical.

\paragraph{Example 2 (Sod shock-tube problem).} In the second example, we consider the Sod shock-tube problem from \cite{sod1978survey}. The
Riemann initial data,
$$
\bm V(x,0)=\left\{\begin{aligned}&(1,0,1)^\top,&&x<0.5,\\&(0.125,0,0.1)^\top,&&\mbox{otherwise},\end{aligned}\right.
$$
are prescribed in the computational domain $[0,1]$ with the free boundary conditions.

We compute the solution until the final time $t=0.2$ on a uniform mesh with $N=200$ and plot it in Figure \ref{fig41} together with the 
exact solution, obtained using the library NUMERICA \cite{toro1999numerica}, and the solution computed by the DF-FV method but without
post-processing. As one can see, the DF-FV solution is oscillation-free and its rather sharply resolved discontinuities are located at the
correct locations thanks to the conservative post-processing. On the contrary, the DF-FV method without the post-processing produces very
large spurious spikes in the $\bm U$-solution (note that these spikes are located exactly where the initial condition was discontinuous) and
a non-oscillatory $\bm V$-solution, which apparently captures a wrong weak solution as expected (see
\cite{abgrall2010comment,hou1994nonconservative}). Therefore, using such $\bm V$-values to update the $\bm U$-variables prevents convergence
to the correct weak solution, despite the conservative nature of the update. In fact, refining the mesh does not lead to convergence of the
$\bm U$ solution, and the spurious spikes persist without decay.
\begin{figure}[ht!]
\centerline{\includegraphics[width=0.32\textwidth]{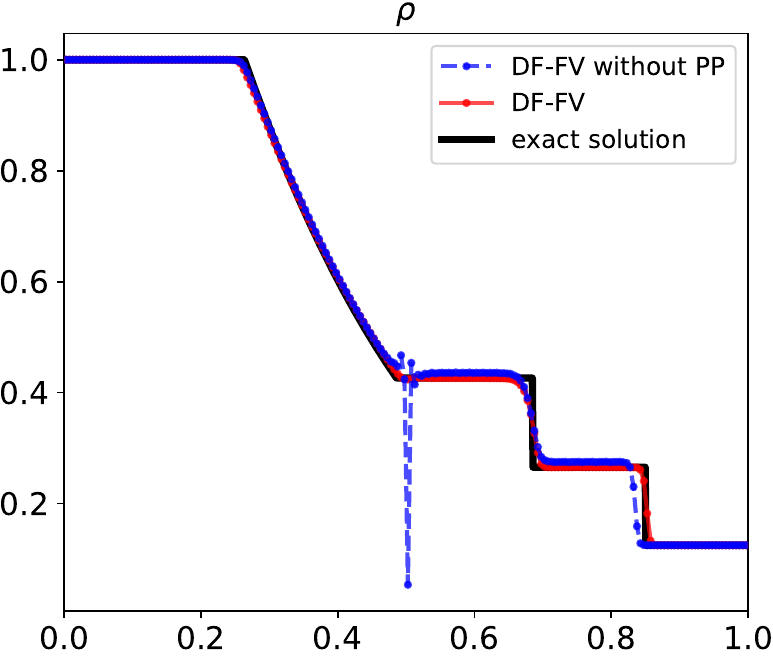}\hspace*{0.2cm}
	    \includegraphics[width=0.32\textwidth]{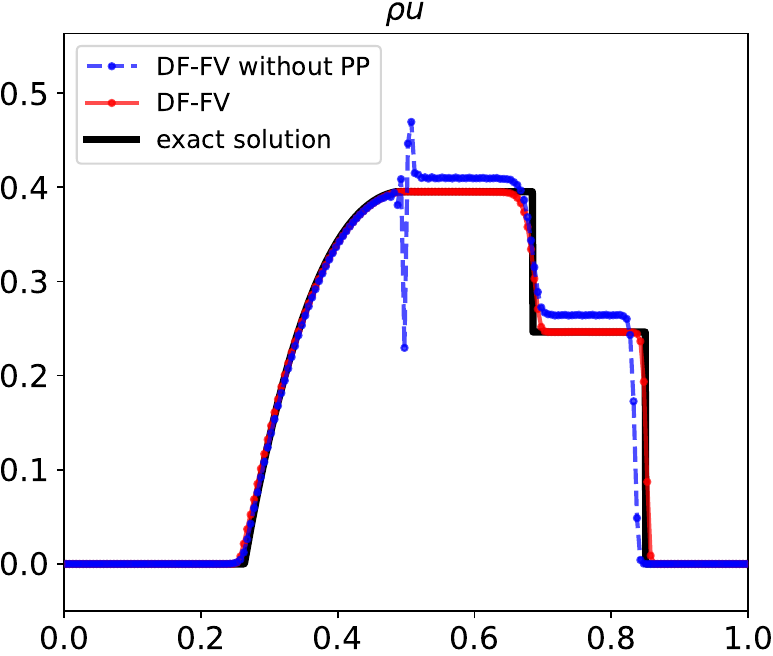}\hspace*{0.2cm}
	    \includegraphics[width=0.32\textwidth]{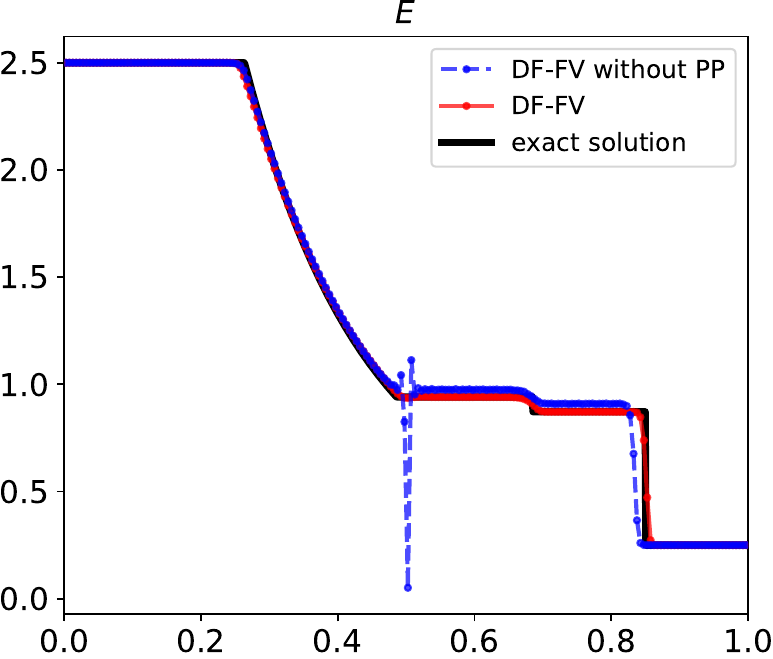}}
\vskip5pt
\centerline{\includegraphics[width=0.32\textwidth]{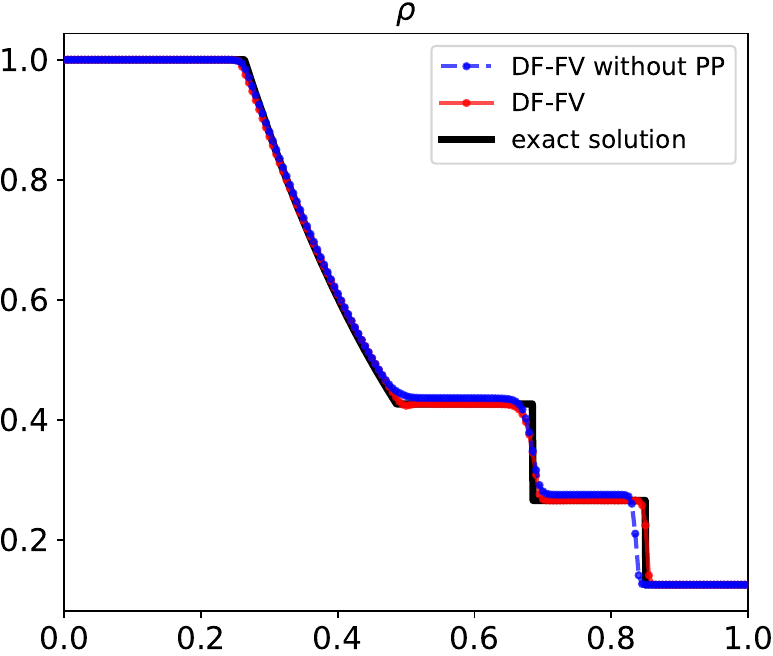}\hspace*{0.2cm}
	    \includegraphics[width=0.32\textwidth]{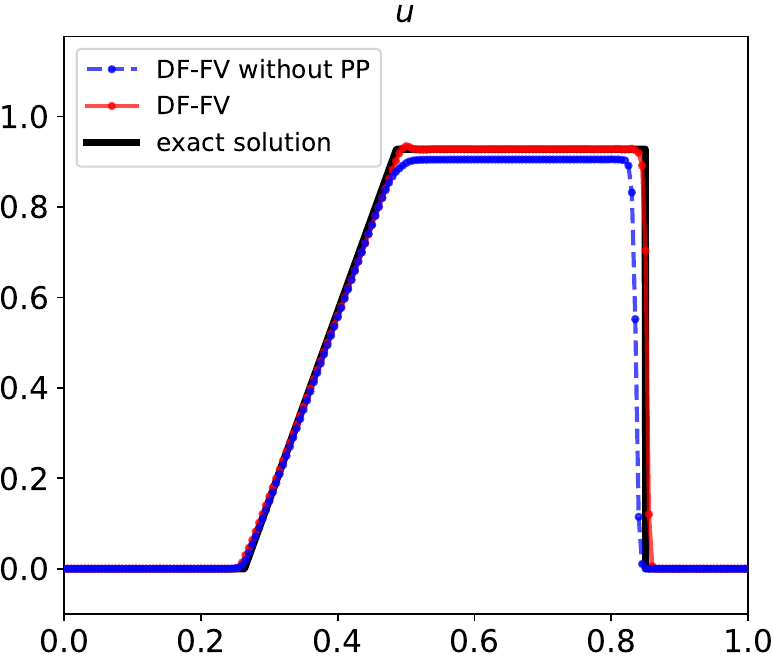}\hspace*{0.2cm}
            \includegraphics[width=0.32\textwidth]{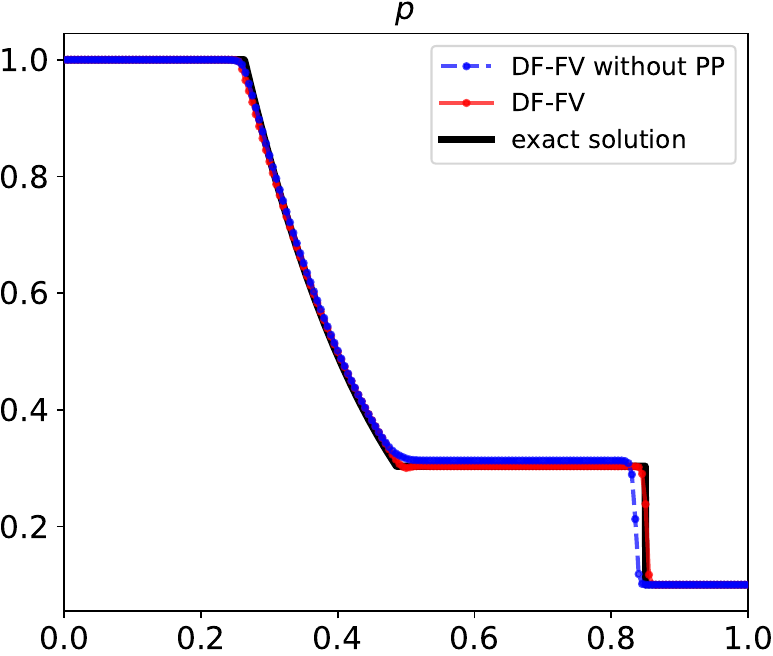}}
\caption{\sf Example 2: $\bm U$-solutions (upper row) and $\bm V$-solutions (lower row) computed by the DF-FV method and the DF-FV method
without the post-processing.\label{fig41}}
\end{figure}

In Figure \ref{fig41f}, we report the comparison between the DF-FV method and CSOC. While the results are comparable, one can see that the
DF-FV method slightly outperforms its counterpart even on this very basic numerical example.
\begin{figure}[ht!]
\centerline{\includegraphics[width=0.32\textwidth]{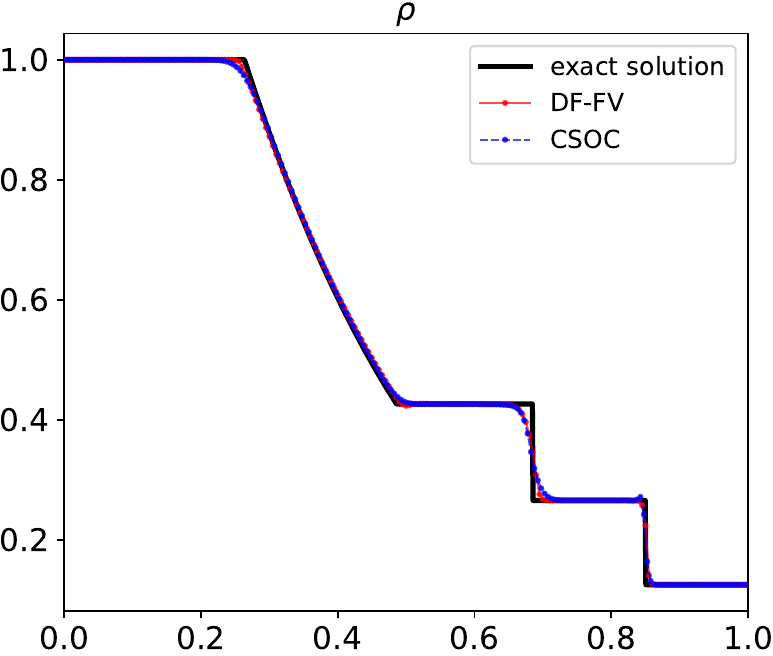}\hspace*{0.2cm}
            \includegraphics[width=0.32\textwidth]{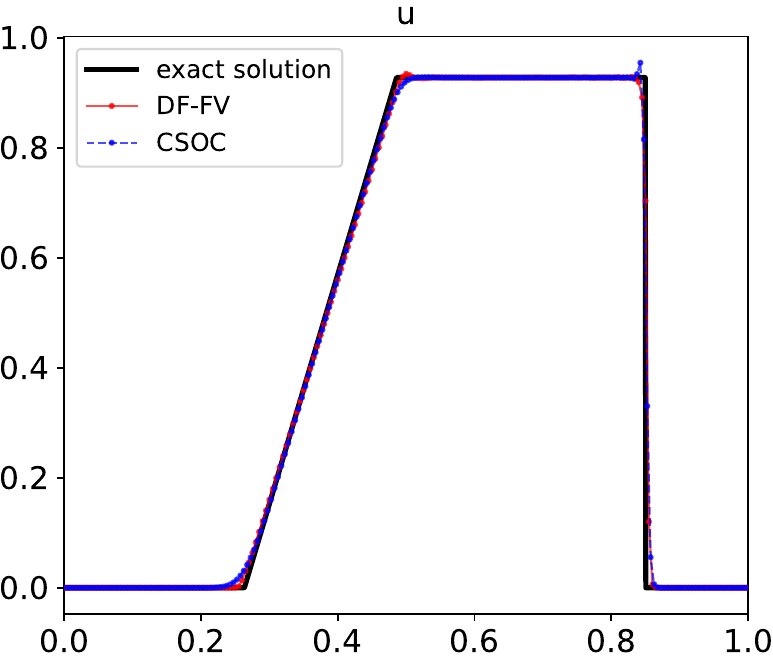}\hspace*{0.2cm}
            \includegraphics[width=0.32\textwidth]{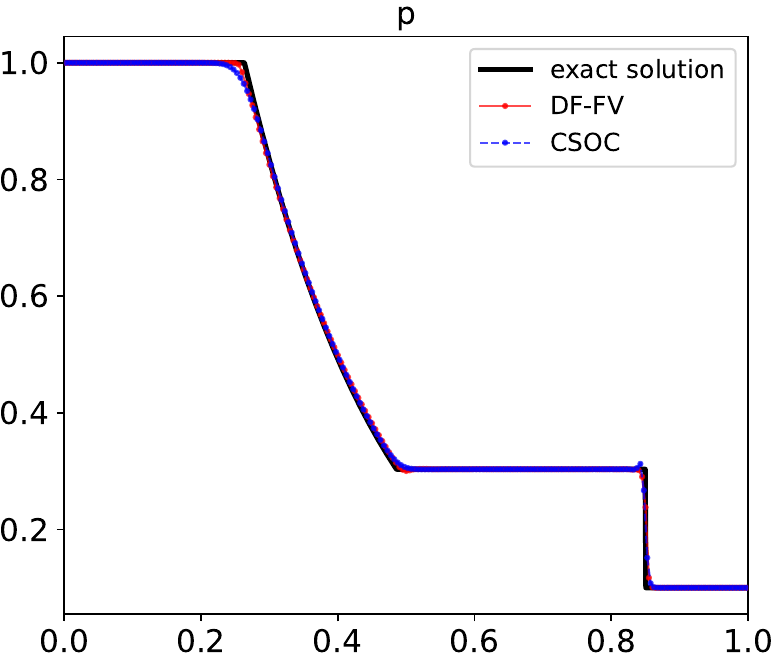}}
\caption{\sf Example 2: Comparison between the DF-FV $\bm V$-solution and CSOC solution.\label{fig41f}}
\end{figure}

\paragraph{Example 3 (double rarefaction problem).} In this problem taken from \cite{ToroBook}, we consider a Riemann problem, whose
solution contains two rarefaction waves, which expand and form a near-vacuum area in the middle of the computational domain $[0,1]$.
The initial conditions,
$$
\bm V(x,0)=\left\{\begin{aligned}&(1,-2,0.4)^\top,&&x<0.5,\\&(1,2,0.4)^\top,&&\mbox{otherwise},\end{aligned}\right.
$$
are supplemented with the free boundary conditions. 

We compute the solution by the DF-FV method and CSOC until the final time $t=0.15$ on a uniform mesh with $N=200$ and plot the obtained
results together with the exact solution, once again generated using the library NUMERICA \cite{toro1999numerica} in Figure \ref{fig42}. As
one can see, both schemes preserve the positivity of the density and pressure and the obtained solutions are oscillation-free. At the same
time, the proposed DF-FV method outperforms the CSOC in the resolution of low-density parts of the computed solutions and near the
rarefaction corners even though the enlarged numerical diffusion attributed to the post-processing oversmears the velocity profile captured
by the DF-FV method.
\begin{figure}[ht!]
\centerline{\includegraphics[width=0.31\textwidth]{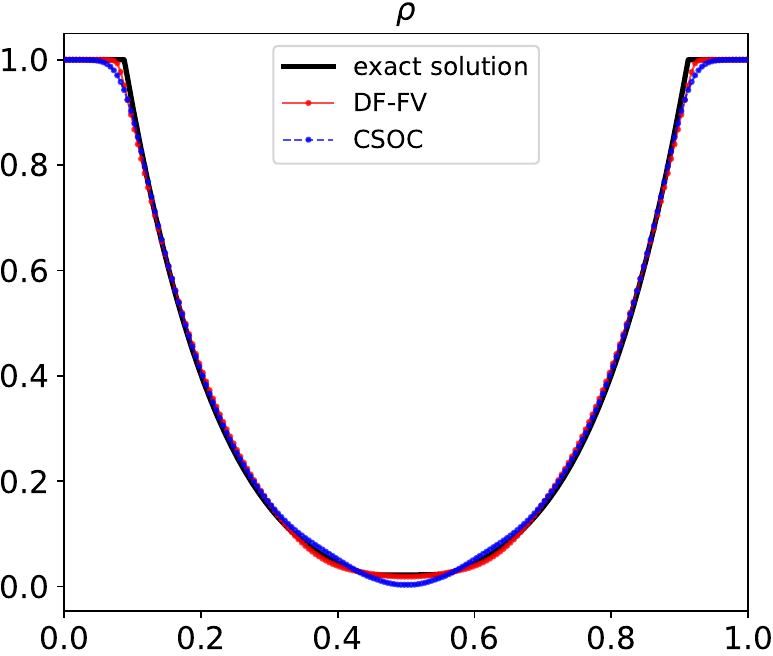}\hspace*{0.2cm}
            \includegraphics[width=0.32\textwidth]{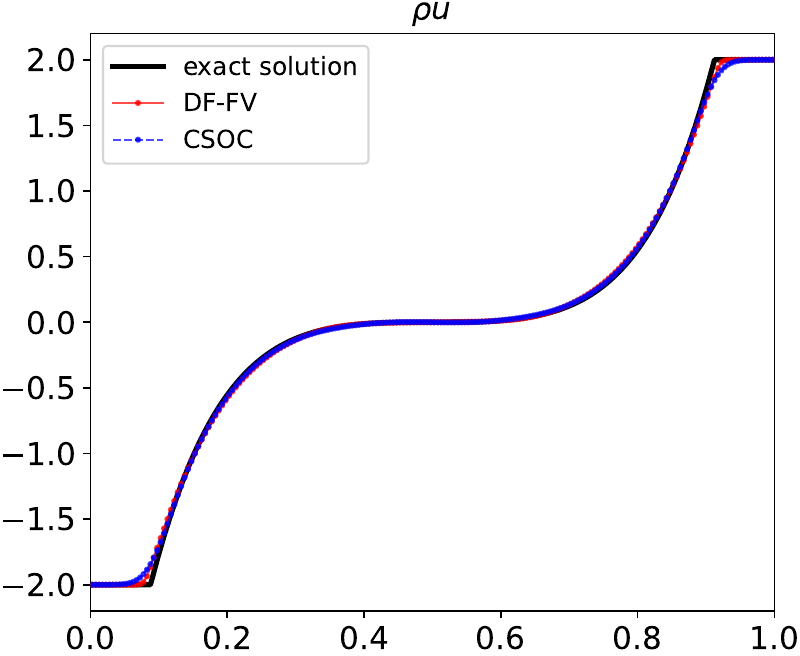}\hspace*{0.2cm}
	    \includegraphics[width=0.31\textwidth]{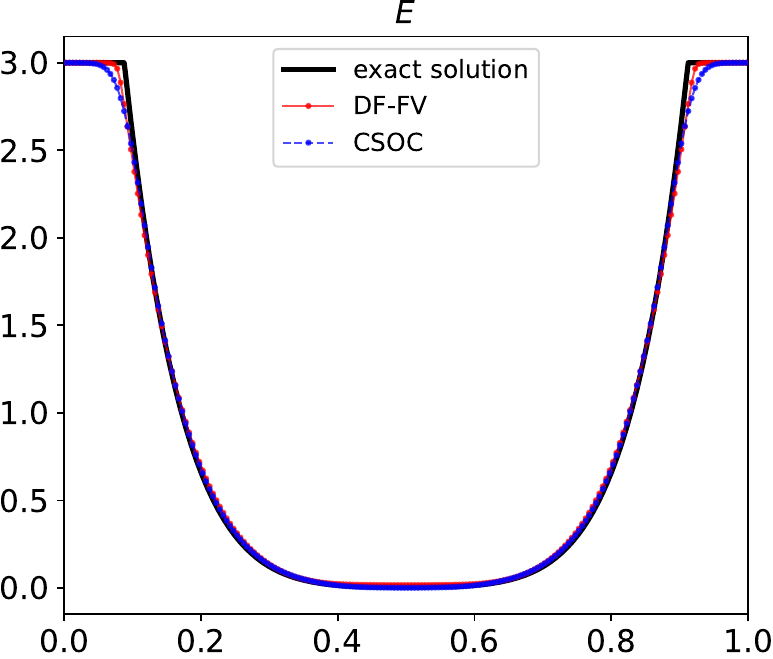}}
\vskip5pt
\centerline{\includegraphics[width=0.31\textwidth]{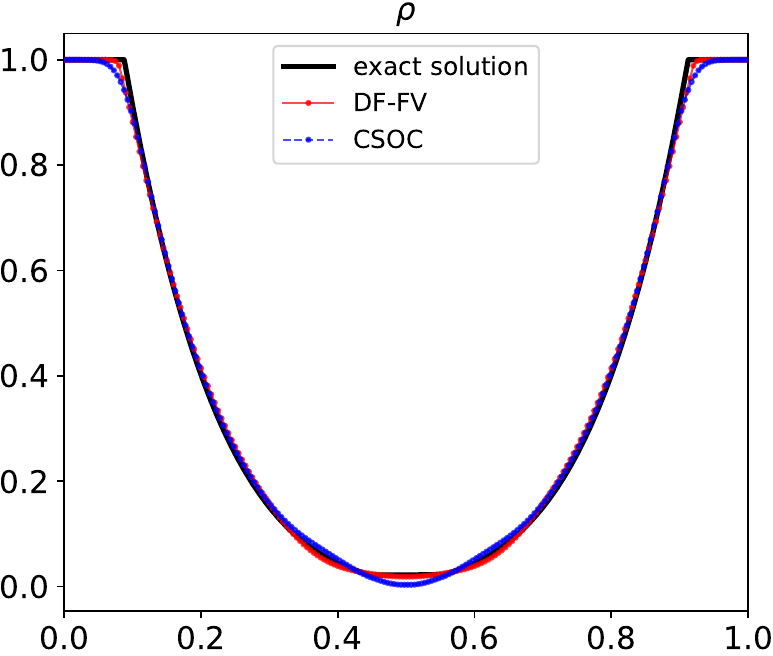}\hspace*{0.2cm}
            \includegraphics[width=0.32\textwidth]{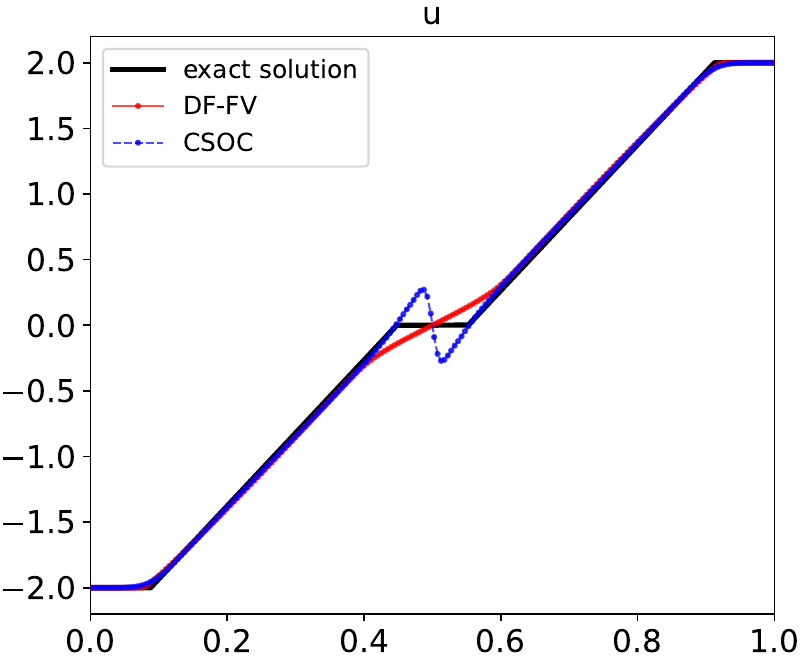}\hspace*{0.2cm}
            \includegraphics[width=0.32\textwidth]{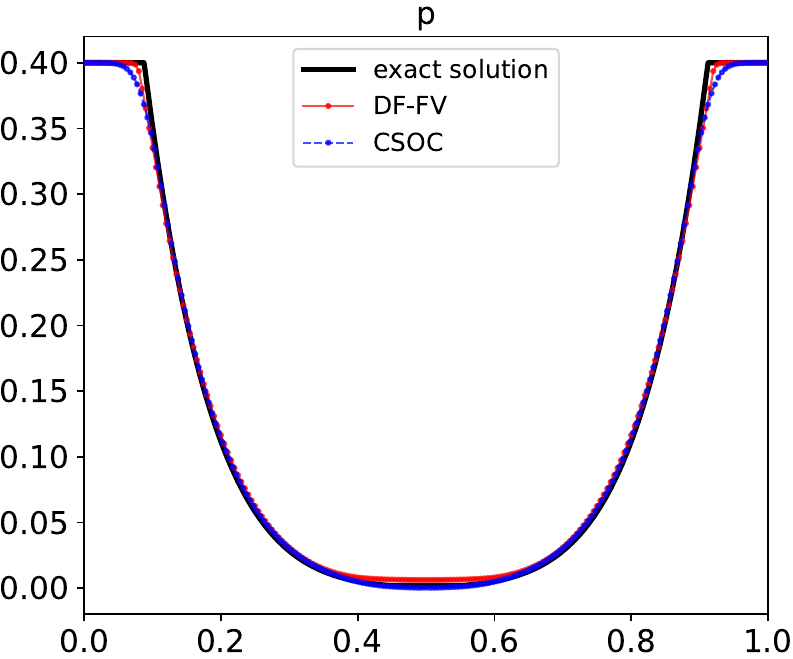}}
\caption{\sf Example 3: $\bm U$-solutions (upper row) and $\bm V$-solutions (lower row) plotted along with the CSOC solution.\label{fig42}}
\end{figure}

\paragraph{Example 4 (shock-turbulence interaction problem).} In this example taken from \cite{shu1989efficient}, a shock interacts with a
turbulent flow characterized by high-frequency oscillations.

The initial data,
$$
\bm V(x,0)=
\left\{\begin{aligned}
&(3.857143,2.629369,10.333333)^\top,&&x<-4,\\
&(1+0.2\sin(5x),0,1)^\top,&&\mbox{otherwise},
\end{aligned}\right.
$$
are prescribed in the computational domain $[-5,5]$ with the inflow boundary conditions at the left boundary and free boundary conditions at
the right one.

We compute the solution by the DF-FV method and CSOC until the final time $t=1.8$ on a uniform mesh with $N=600$ and report the obtained
results in Figure \ref{fig43}. Since in this example no exact solution is available, we compute the reference solution by a second-order
semi-discrete CU scheme from \cite{KLin} on a much finer uniform mesh consisting of $200000$ cells. As one can see, the solution structures
are correctly resolved by both of the studied schemes, but the resolution of the smooth, oscillating parts of the solution achieved by the
DF-FV method is substantially higher, as can be further seen in Figure \ref{fig45f}, where a zoom at the spatial interval $[0,2.5]$ is
shown.
\begin{figure}[ht!]
\centerline{\includegraphics[width=0.32\textwidth]{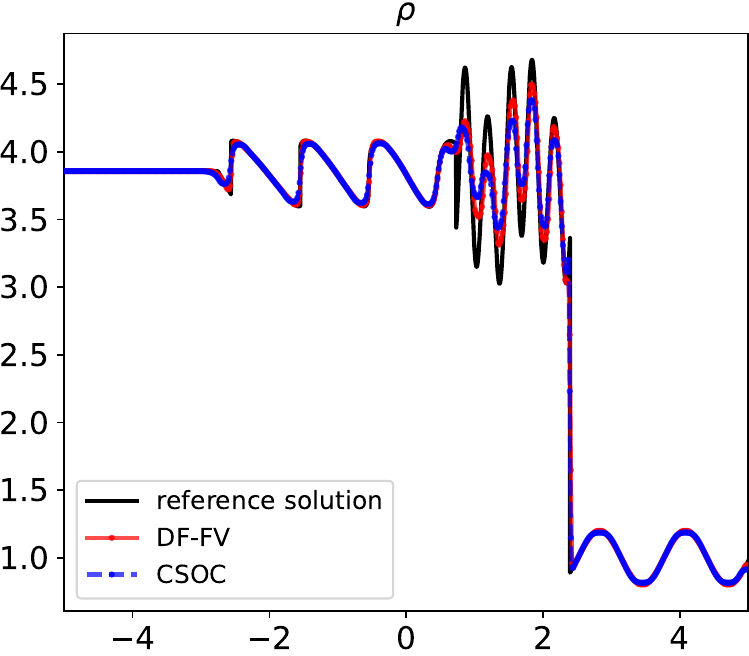}\hspace*{0.2cm}
            \includegraphics[width=0.32\textwidth]{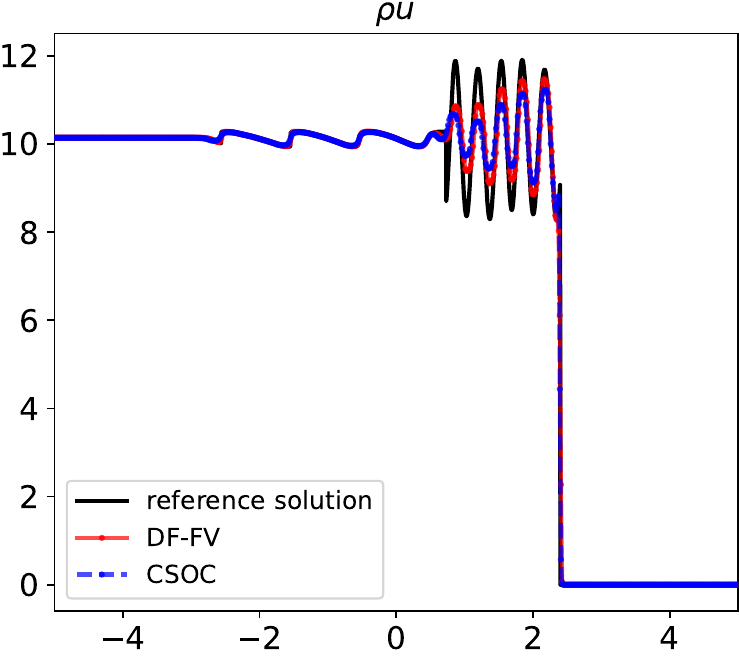}\hspace*{0.2cm}
	    \includegraphics[width=0.32\textwidth]{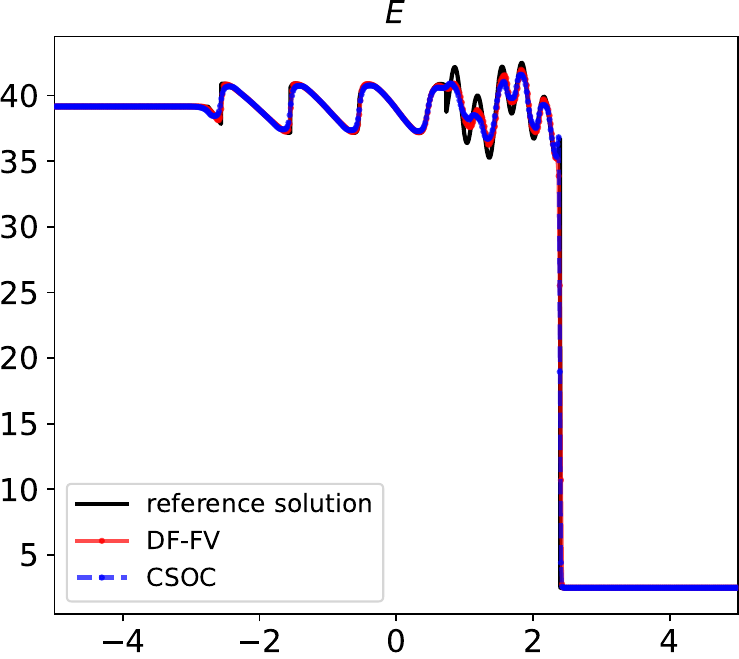}}
\vskip5pt
\centerline{\includegraphics[width=0.32\textwidth]{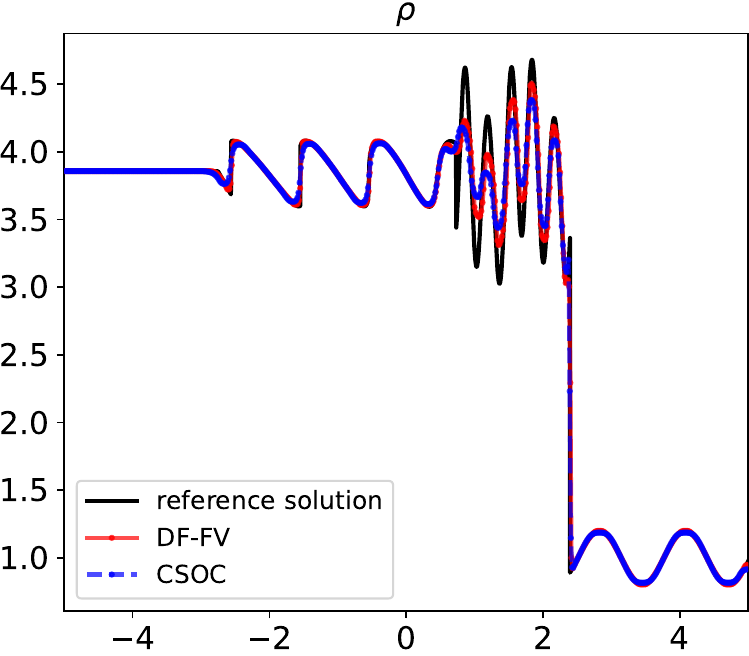}\hspace*{0.2cm}
            \includegraphics[width=0.32\textwidth]{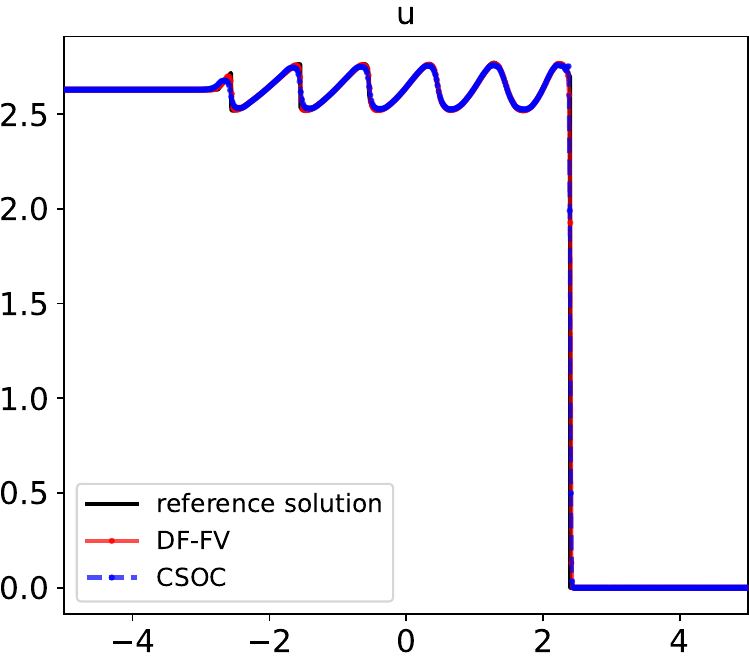}\hspace*{0.2cm}
            \includegraphics[width=0.32\textwidth]{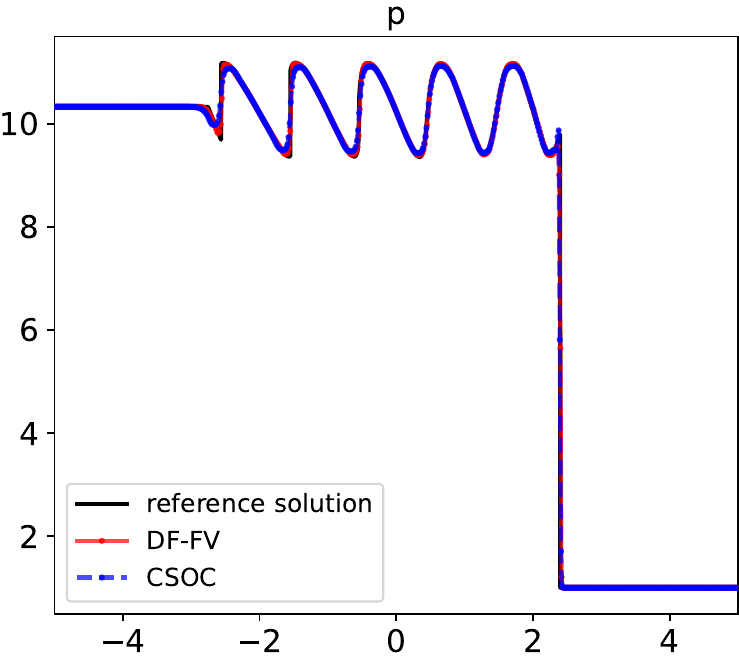}}
\caption{\sf Example 4: $\bm U$-solutions (upper row) and $\bm V$-solutions (lower row) plotted along with the CSOC solution.\label{fig43}}
\end{figure}
\begin{figure}[ht!]
\centerline{\includegraphics[trim=0.3cm 0.3cm 0.2cm 0.3cm, clip, width=0.43\textwidth]{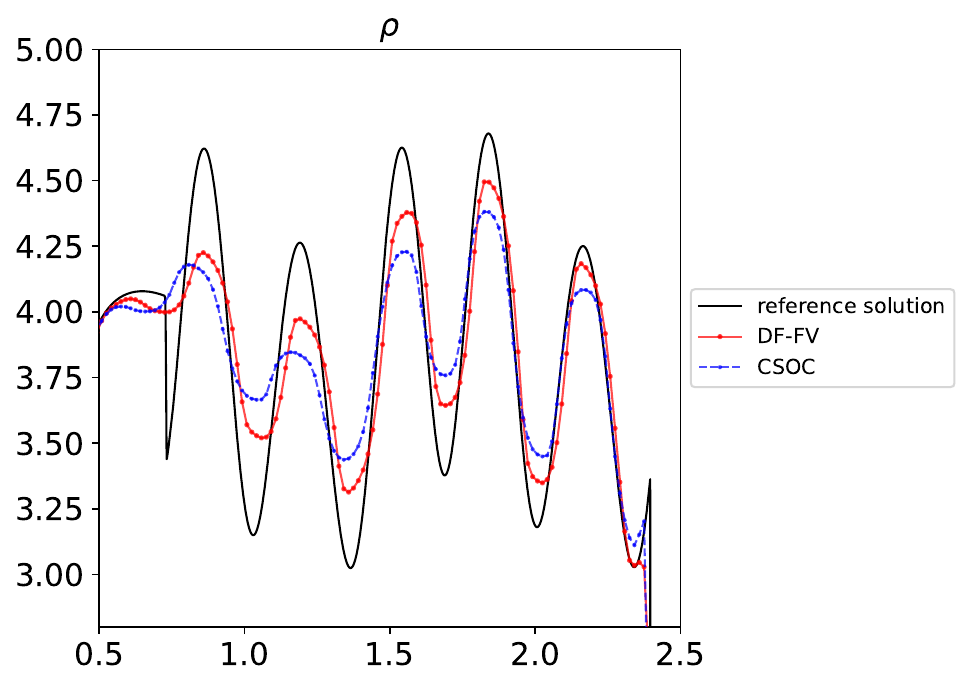}}
\caption{\sf The same as in the upper left panel in Figure \ref{fig43}, but zoomed at $x\in[0,2.5]$.\label{fig45f}}
\end{figure}

\paragraph{Example 5 (Woodward-Colella problem).} In this problem, which was introduced in \cite{woodward1984numerical}, the initial 
conditions,
$$
\bm V(x,0)=
\left\{\begin{aligned}
&(1,0,10^3)^\top,&&x<0.1,\\
&(1,0,10^2)^\top,&&x>0.9,\\
&(1,0,10^{-2})^\top,&&\mbox{otherwise},
\end{aligned}\right.
$$
are prescribed in the computational domain $[0,1]$ with solid wall boundary conditions.

In Figure \ref{fig45}, we plot the solutions computed by the DF-FV method and CSOC at the final time $t=0.038$ on a uniform mesh with
$N=400$ along with the reference solution generated by a second-order semi-discrete CU scheme from \cite{KLin} with reconstruction of
characteristic variables and the same time discretization on a much finer uniform mesh consisting of $200000$ cells. The obtained results
demonstrate the ability of the proposed DF-FV method to capture strong discontinuities. It also clearly shows the superiority of our method
compared to the CSOC.
\begin{figure}[ht!]
\centerline{\includegraphics[width=0.30\textwidth]{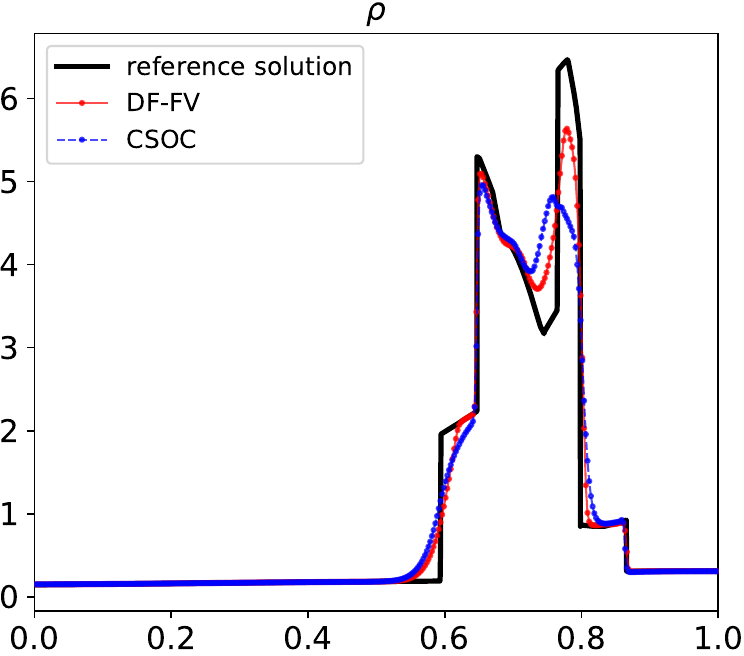}\hspace*{0.2cm}
            \includegraphics[width=0.31\textwidth]{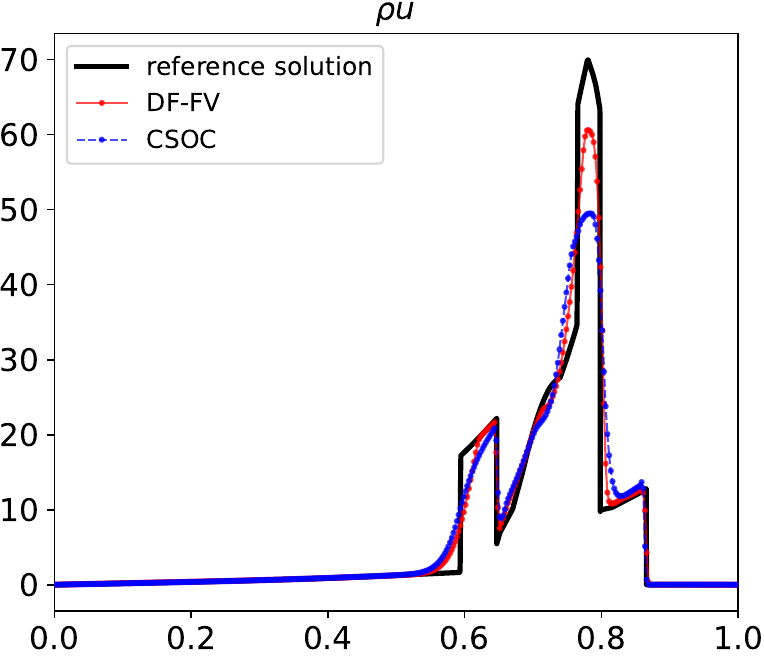}\hspace*{0.2cm}
            \includegraphics[width=0.33\textwidth]{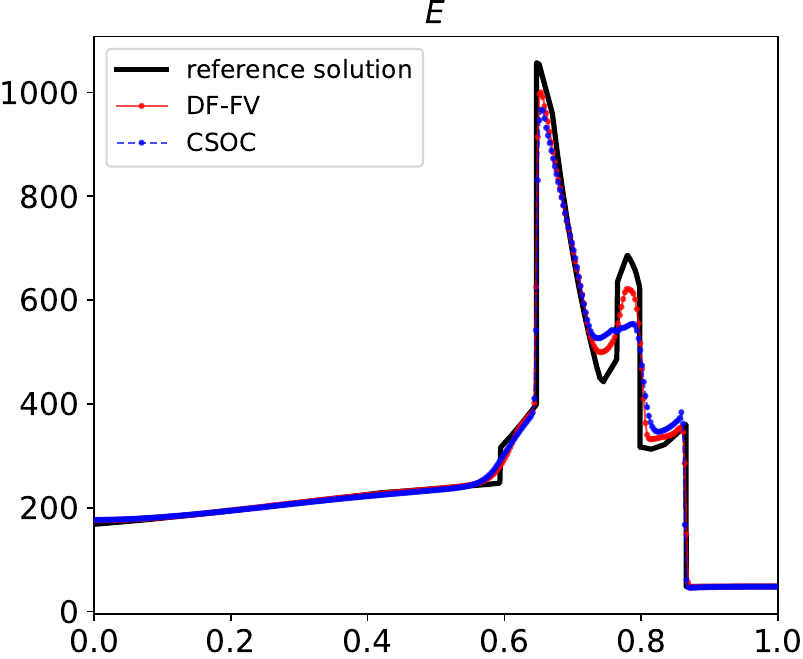}}
\vskip5pt
\centerline{\includegraphics[width=0.30\textwidth]{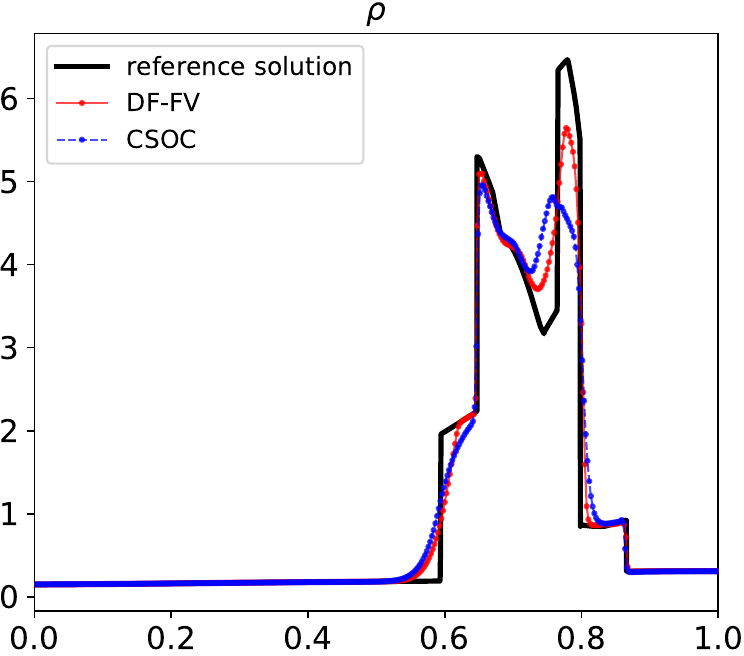}\hspace*{0.25cm}
            \includegraphics[width=0.31\textwidth]{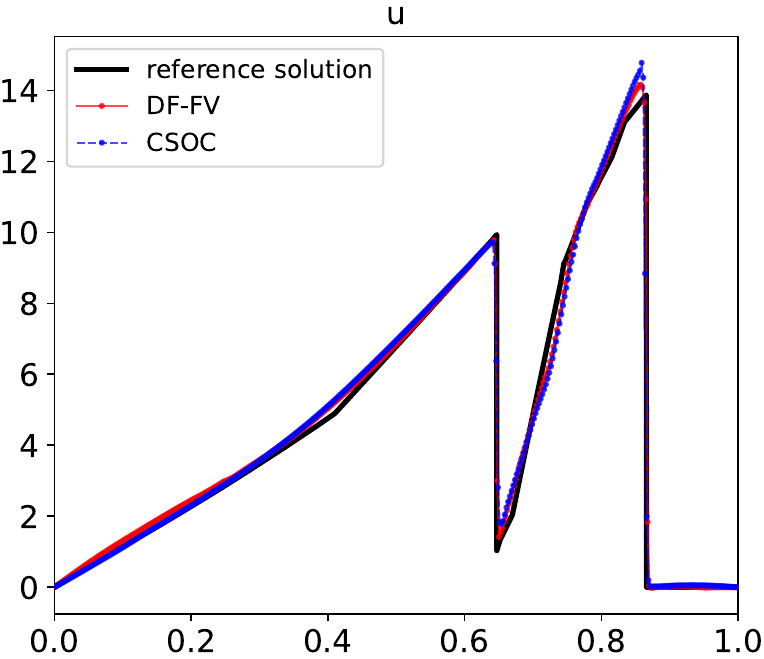}\hspace*{0.25cm}
            \includegraphics[width=0.32\textwidth]{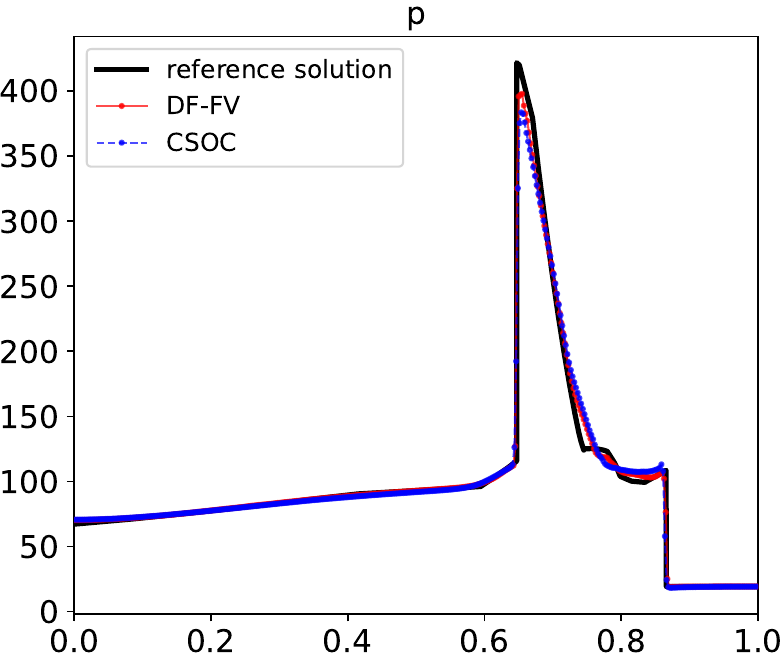}}
\caption{\sf Example 5: $\bm U$-solutions (upper row) and $\bm V$-solutions (lower row) plotted along with the CSOC solution.\label{fig45}}
\end{figure}

\paragraph{Example 6 (explosion problem).} This problem, taken from \cite{ToroBook}, is a multidimensional extension of the Sod shock-tube problem considered in Example 2.

We take the following initial conditions:
$$
\bm V(x,y,0)=\left\{\begin{aligned}
&(1,0,0,1)^\top,&&\sqrt{x^2+y^2}<0.4,\\
&(0.125,0,0,0.1)^\top,&&\mbox{otherwise},
\end{aligned}\right.
$$
which are prescribed in the computational domain $[-1,1]\times[-1,1]$ with the free boundary conditions.

We compute the solution until the final time $t=0.25$ on a uniform mesh with $N_x=N_y=400$. In Figure \ref{fig46}, we report
three-dimensional plots of the density- and energy-components of the $\bm U$-solution and the density- and pressure-components of the
$\bm V^x$-solution (we do not show the $\bm V^y$-solution as it is practically the same as the $\bm V^x$-one). The 1-D slices of $\rho$,
$\sqrt{(\rho u)^2+(\rho v)^2}$, and $E$ from the $\bm U$-solution and of $\rho$, $\sqrt{u^2+v^2}$, and $p$ from the $\bm V^x$-solution 
along the line $y=x$ are shown in Figure \ref{fig47} (the reference solution plotted there has been obtained by a second-order semi-discrete
FV scheme with reconstruction of characteristic variables and exact Riemann solver numerical flux on a much finer uniform mesh consisting of
$3000\times3000$ cells). As one can see, the numerical solution does not exhibit significant spurious oscillations or anomalous features,
and the solution is sharply captured. One can also observe that the radial symmetry is preserved.
\newcommand{\trimfig}[1]{\includegraphics[width=0.35\textwidth, trim=140 260 190 250, clip]{#1}}
\begin{figure}[ht!]
\centerline{\trimfig{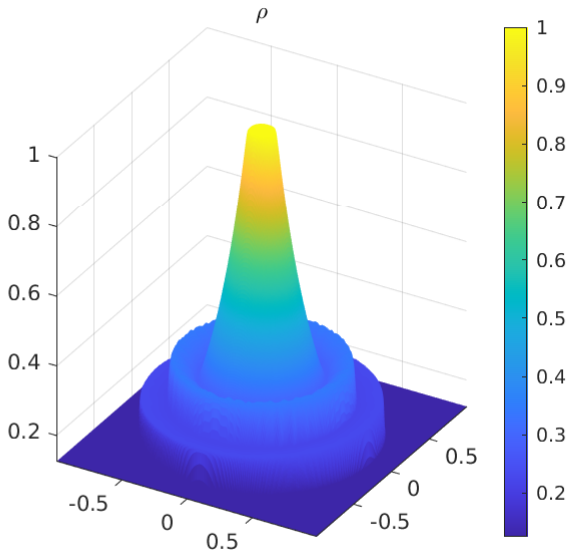}\qquad\trimfig{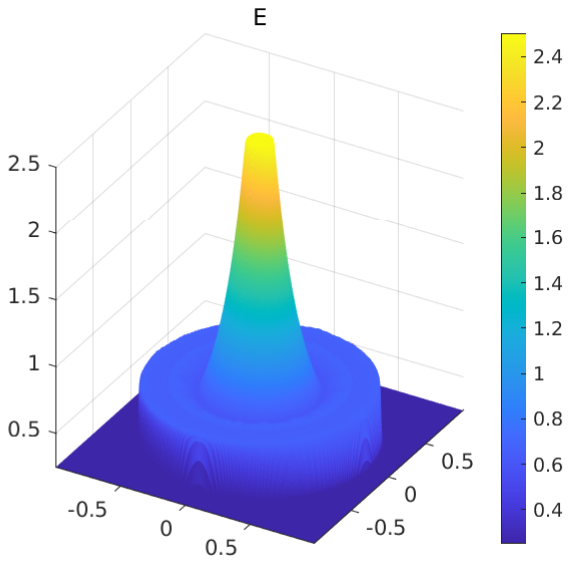}}
\vskip5pt
\centerline{\trimfig{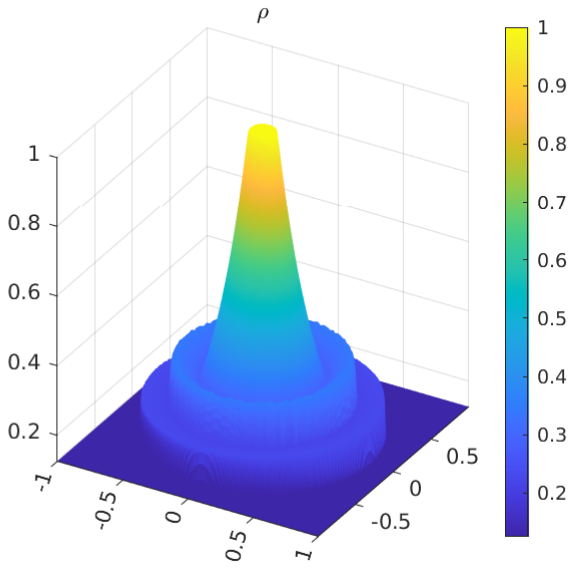}\qquad\trimfig{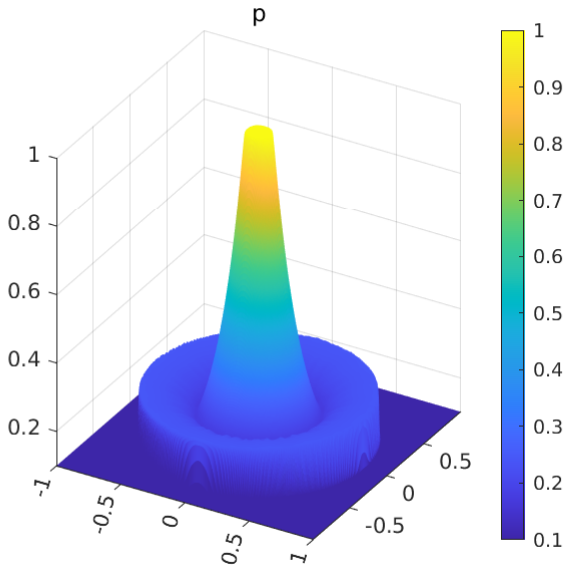}}
\caption{\sf Example 6: Density- and energy-components of the $\bm U$-solution (upper row) and density- and pressure-components of the
$\bm V^x$-solution (lower row).\label{fig46}}
\end{figure}
\begin{figure}[ht!]
\centerline{\includegraphics[width=0.31\textwidth]{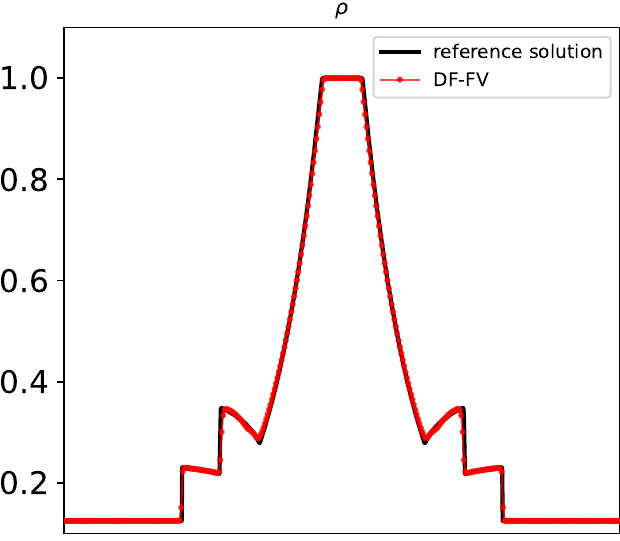}\hspace*{0.3cm}
            \includegraphics[width=0.31\textwidth]{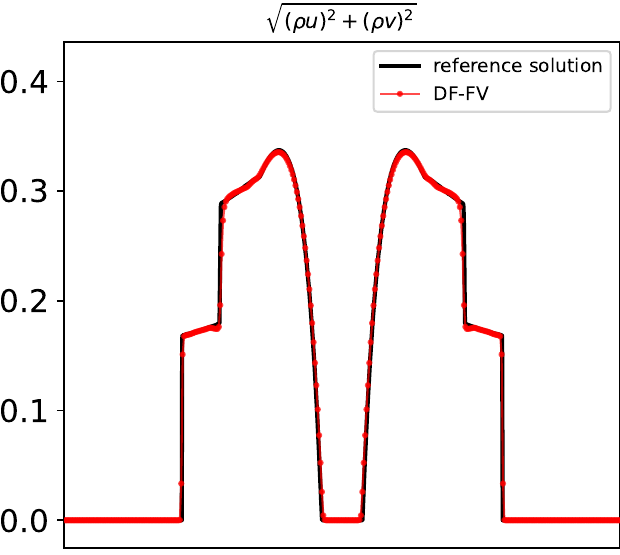}\hspace*{0.3cm}
            \includegraphics[width=0.31\textwidth]{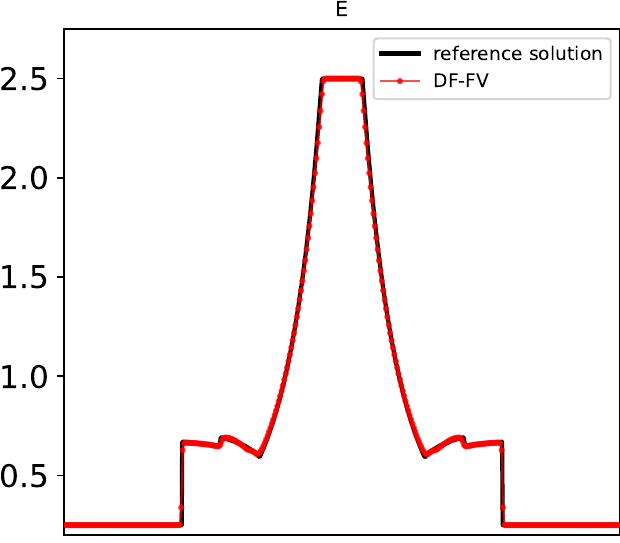}}
\vskip5pt
\centerline{\includegraphics[width=0.31\textwidth]{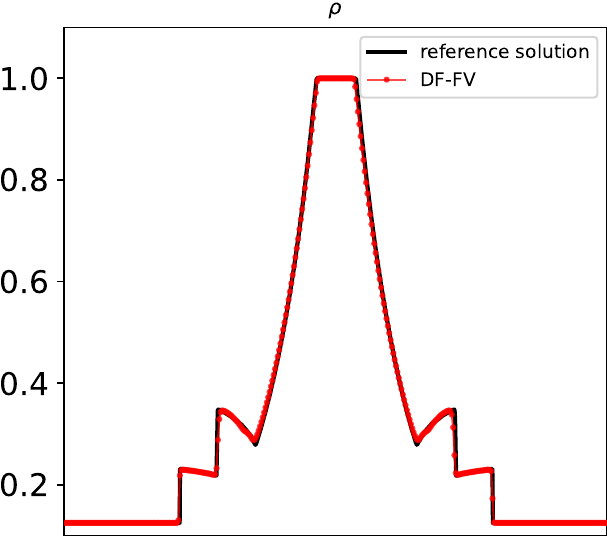}\hspace*{0.3cm}
            \includegraphics[width=0.31\textwidth]{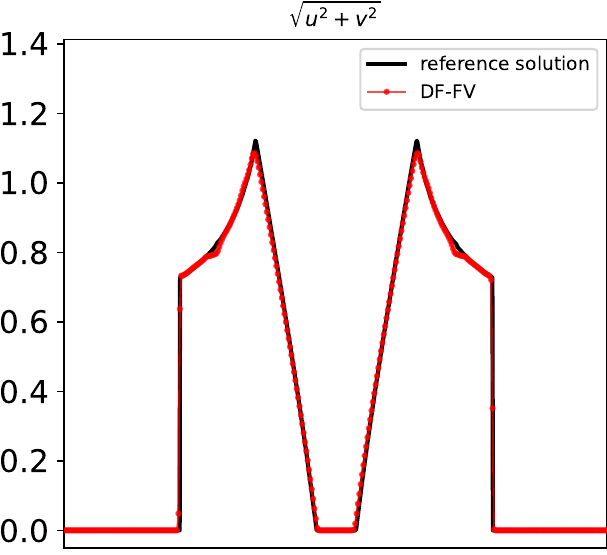}\hspace*{0.3cm}
	    \includegraphics[width=0.31\textwidth]{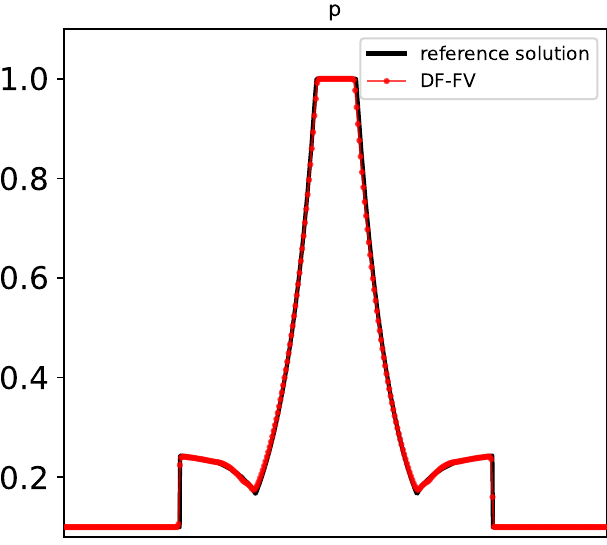}}
\caption{\sf Example 6: 1-D slices along $y=x$ for the $\bm U$-solution (upper row) and $\bm V^x$-solution (lower row).\label{fig47}}
\end{figure}

\paragraph{Example 7 (shock-vortex interaction).} 
In this test, which was studied in, e.g.,\cite{DZLD,geisenhofer2019discontinuous,RCD}, we consider a moving vortex with Mach number
$M_v:=0.9$ interacting with a stationary shock with Mach number $M_s:=1.5$. Such an interaction gives rise to complex flow structures,
making this numerical test challenging for higher-order numerical schemes. 

A schematic of the initial setup is shown in Figure \ref{fig48f}. Specifically, the computational domain $[0,2]\times[0,1]$ is divided 
into two main subdomains by the vertical line $x=0.5$. The vortex is initially located within a circular area centered on 
$(x_c,y_c):=(0.25,0.5)$ and occupies regions I and II: the first one is a circle of radius $a=0.075$ and the latter one is a concentrical
annulus with inner radius $a$ and outer radius $b=0.175$. The states in regions III and IV correspond to a stationary shock. In particular, 
given the left constant state from region III,
$$
\rho_{\rm III}=1,\quad u_{\rm III}=\sqrt{\gamma}M_s,\quad v_{\rm III}=0,\quad p_{\rm III}=1,
$$
the right state, $(\rho_{\rm IV},u_{\rm IV},v_{\rm IV},p_{\rm IV})$, can be easily computed through the Rankine-Hugoniot conditions (see
\cite[Section 3.1.3]{ToroBook}), leading to
$$
\rho_{\rm IV}=\frac{(\gamma+1)M_s^2}{(\gamma-1)M_s^2+2}\rho_{\rm III},\quad u_{\rm IV}=\frac{(\gamma-1)M_s^2+2}{(\gamma+1)M_s^2}u_{\rm III}
\quad v_{\rm IV}=0,\quad p_{\rm IV}=\frac{2\gamma M_s^2-(\gamma-1)}{\gamma+1}p_{\rm III}.
$$
\begin{figure}[ht!]
\centering
\begin{tikzpicture}[scale=7]
\def\angleA{-60}   
\def\angleB{145}    
\def\angleR{45}   
\def\rA{0.09}
\def\rB{0.18}
\def\rR{0.16}
\def\bottom{0.2}
\def\top{0.8}
\coordinate (C) at (0.25,0.5);
\draw[thick] (0,\bottom) rectangle (1.35,\top);
\draw[very thick] (0.5,\bottom) -- (0.5,\top);
\draw[line width=0.5mm] (C) circle (\rA); 
\draw[line width=0.5mm] (C) circle (\rB); 
\draw[dashed] (C) circle (\rR);  
\draw[->] (C) -- ++({\angleA}:\rA)
node[pos=0.7, anchor=south] {\(a\)};
\draw[->] (C) -- ++({\angleB}:\rB)
node[pos=0.6, anchor=south] {\(b\)};
\draw[->] (C) -- ++({\angleR}:\rR)
node[pos=0.6, anchor=south] {\(r\)};
\draw[dashed] (C) -- (0.25, \bottom); 
\draw[dashed] (C) -- (0, 0.5);  
\node[anchor=east] at (0,\top) {$y=1$};
\node[anchor=east] at (0,\bottom) {$y=0$};
\node[anchor=north] at (0,\bottom) {$x=0$};
\node[anchor=north] at (0.5,\bottom) {$x=0.5$};
\node[anchor=north] at (1.35,\bottom) {$x=2$};
\node[anchor=south] at (0.5,\top+0.03) {stationary};
\node[anchor=south] at (0.5,\top) {shock};
\node[anchor=north] at (0.25,\bottom) {$x=0.25$};
\node[anchor=east] at (0,0.5) {$y=0.5$};
\node[scale=0.8] at (0.22, 0.46) {\textbf{I}};
\node[scale=0.8] at (0.18, 0.4) {\textbf{II}};
\node[scale=0.8] at (0.13, 0.31) {\textbf{III}};
\node[scale=0.8] at (1.0, 0.5) {\textbf{IV}};
\draw[dashed] (0.25,0.5) -- (0.25+\rB,0.5); 
\def\rTheta{0.11} 
\draw[->] ($(C)+(\rTheta,0)$) arc[start angle=0, end angle=\angleR, radius=\rTheta];
\node at ($(C)+({\angleR/2}:1.2*\rTheta)$) {\(\vartheta\)};		
\end{tikzpicture}
\caption{\sf Example 7: Sketch of the simulation set-up.}
\label{fig48f}
\end{figure}

The velocity profile in vortex regions I and II is given, in terms of radial coordinates $(r,\vartheta)$ with respect to the center 
$(x_c,y_c)$:
$$
u(r,\vartheta)=u_{\rm III}-v_\vartheta\sin\vartheta,\quad v(r,\vartheta)=v_{\rm III}+v_\vartheta\cos{\vartheta},
$$
where 
$$
v_\vartheta:=\left\{\begin{aligned}
&v_m\,\frac{r}{a},&&r\le a,\\
&v_m\,\frac{a}{a^2-b^2}\left(r-\frac{b^2}{r}\right),&&a<r<b,\\
&0,&&r\ge b,\\
\end{aligned}\right.
$$
with $v_m:=M_v\sqrt{\gamma}$ being the the maximal angular velocity.

Density and pressure profiles inside the vortex are obtained by imposing a balance between centripetal force and pressure gradients 
\cite{RCD}, resulting in
$$
p=p_{\rm III}\left(\frac{T}{T_{\rm III}}\right)^{\frac{\gamma}{\gamma-1}},\quad
\rho=\rho_{\rm III}\left(\frac{T}{T_{\rm III}}\right)^{\frac{1}{\gamma-1}},
$$
where $T_{\rm III}=\frac{p_{\rm III}}{\rho_{\rm III}R}$ is the constant temperature associated with the state of region III, with
$R=287$ J/kg-K being the specific gas constant of the fluid, and $T(r)$ is the temperature within the vortex:
$$
T(r)=\left\{\begin{aligned}
&A+\frac{\gamma-1}{R\gamma}\frac{v_m^2}{a^2}\frac{r^2}{2},&&r\le a,\\
&B+\frac{\gamma-1}{R\gamma}v_m^2\frac{a^2}{(a^2-b^2)^2}\left(\frac{r^2}{2}-2b^2\ln r-b^4\frac{r^{-2}}{2}\right),&&a<r<b,\\
&T_{\rm III},&&r\ge b,
\end{aligned}\right.
$$
where
$$
\begin{aligned}
B&=T_{\rm III}-\frac{\gamma-1}{R\gamma}v_m^2\frac{a^2}{(a^2-b^2)^2}\left(\frac{b^2}{2}-2b^2\ln b-b^4\frac{b^{-2}}{2}\right),\\[0.3ex]
A&=B+\frac{\gamma-1}{R\gamma}v_m^2\frac{a^2}{(a^2-b^2)^2}\left(\frac{a^2}{2}-2b^2\ln a-b^4\frac{a^{-2}}{2}\right)-
\frac{\gamma-1}{R\gamma}\frac{v_m^2}{2}.	
\end{aligned}
$$

Figure \ref{fig48} shows a Schlieren image of the numerical results displaying the magnitude of the density gradient field,
$\|\nabla\rho\|_2$, of the $\bm U$-solution computed on a uniform mesh with $N_x=1200$ and $N_y=601$ at time $t=0.7$ along with the
reference solution obtained with the help of a sevent-order WENO-DeC scheme from \cite{micalizzitoro2024,micalizzi2025algorithms}, which was
implemented using with reconstruction of characteristic variables and exact Riemann problem solver on a uniform $800\times401$ mesh. In this
figure, we have used the following shading function:
$$
\exp{\left(-\frac{K\|\nabla\rho\|_2}{\max\|\nabla\rho\|_2}\right)},\quad K=80,
$$
where the numerical density derivatives are computed using central differencing.
\begin{figure}[ht!]
\centerline{\includegraphics[width=0.7\textwidth]{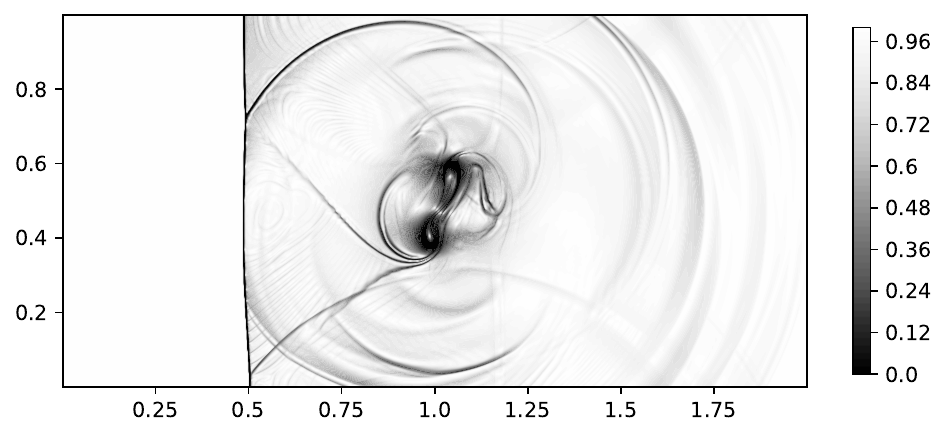}}
\vskip5pt
\centerline{\includegraphics[width=0.7\textwidth]{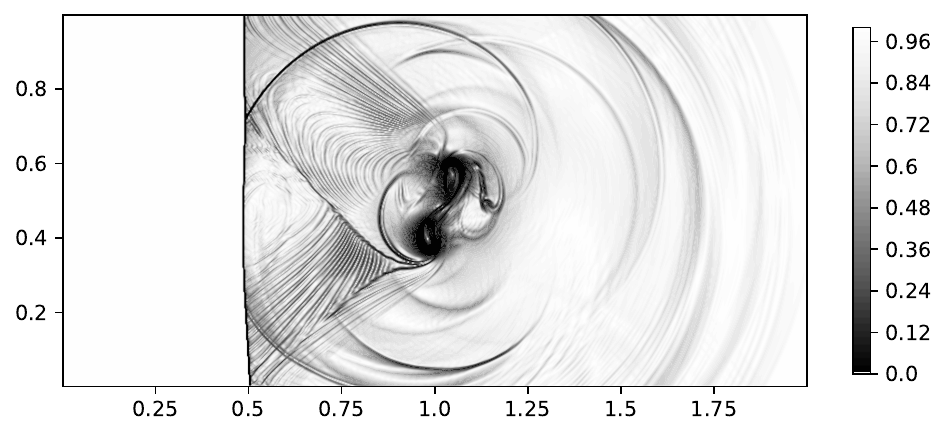}}
\caption{\sf Example 7: Schlieren image of the density gradient of the DF-FV $\bm U$-solution (top) along with the reference solution
(bottom).\label{fig48}}
\end{figure}

Clearly, the reference solution is more accurate and detailed than the DF-FV one. However, as one can see, all the relevant flow features
are correctly captured by the DF-FV method. We note that the computed results also agree well with those obtained in the literature; see,
e.g., \cite{geisenhofer2019discontinuous}.

\paragraph{Example 8 (2-D Riemann problem).} 
In the last example, we consider the 2-D Riemann problem (Configuration 3) from \cite{kurganov2002solution} (also see \cite{chu2025new}). In
the computational domain $[0,1.2]\times[0,1.2]$ with transmissive boundary conditions, we prescribe the following initial conditions:
$$
\bm V(x,y,0)=\begin{cases}(1.5,0,0,1.5)^\top,&x>1,~y>1,\\(0.5323,1.206,0,0.3)^\top,&x<1,~y>1,\\
(0.138,1.206,1.206,0.029)^\top,&x<1,~y<1,\\(0.5323,0,1.206,0.3)^\top,&x>1,~y<1,
\end{cases}
$$
and we run the simulations on a uniform mesh with $N_x=N_y=1000$ until the final time $t=1$. The obtained density component of the
$\bm U$-solution is plotted in Figure \ref{fig411}. As one can see, the obtained result is consistent with that reported in
\cite{chu2025new}, illustrating the ability of the DF-FV method to capture complex flow features of this benchmark.
\begin{figure}[ht!]
\centerline{\includegraphics[width=0.37\textwidth]{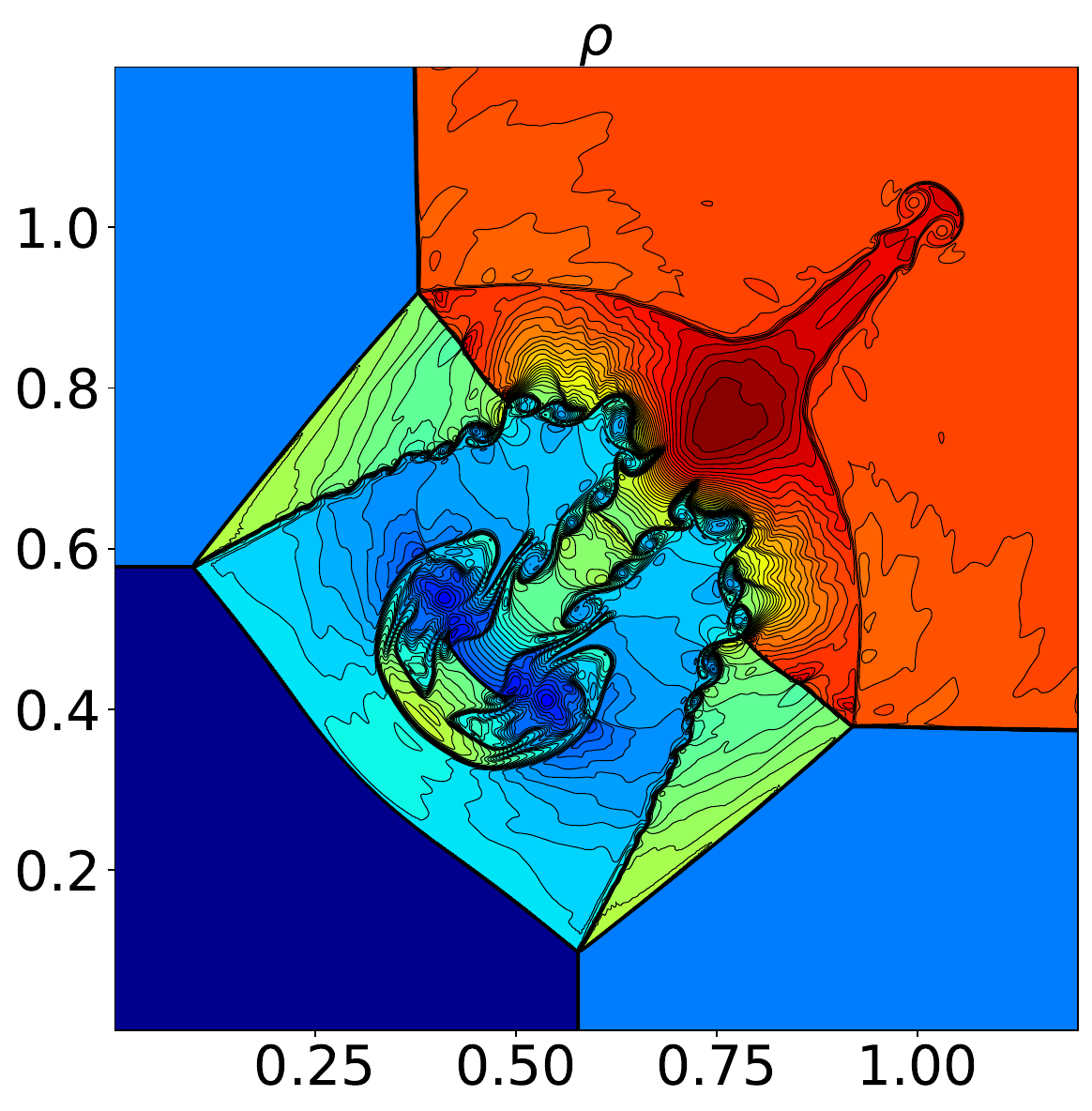}}
\caption{\sf Example 8: Density-component of the $\bm U$-solution.\label{fig411}}
\end{figure}

\section{Conclusions}\label{sec5}
In this paper, we have introduced new methods for one- and two-dimensional hyperbolic systems of conservation laws, for which we consider
two different formulations of the studied systems (the original conservative formulation and a primitive one containing nonconservative
products), and discretize them on overlapping staggered meshes using two different numerical methods. Both the conservative and primitive
variables are evolved in time using second-order semi-discrete finite-volume (FV) methods. The nonconservative system is discretized by a
path-conservative central-upwind scheme, and its solution is used to evaluate very simple numerical fluxes for the original conservative
system. The nonlinear stability of the resulting DF-FV methods is enforced with the help of a post-processing, which also guarantees a
conservative coupling between the two sets of variables. The performance of the proposed methods has been demonstrated on a number of
benchmarks.

The introduced DF-FV methods share an important feature with the AF formulation presented in \cite{Abg23}---the reliance on
extra degrees of freedom used to discretize a nonconservative primitive formulation of the governing equations. On the other hand, several
differences exist between the proposed approach and AF schemes.

\smallskip
\noindent
$\bullet$ While DF-FV methods make use of cell averages of conserved and primitive variables on overlapping grids, AF schemes consider cell
averages of the conserved variables and point values (either of conserved variables as in \cite{eymann2011active,eymann2013multidimensional}
or primitive variables as in \cite{Abg23}) at cell interfaces.

\smallskip
\noindent
$\bullet$ In the DF-FV methods, within each time evolution step, the primitive cell averages are evolved independently from the conserved
ones, differently from the AF approach, in which the update of the point values makes an explicit use of the cell averages. This creates the
need, within the DF-FV framework, for a suitable post-processing to reinstate the necessary coupling between the considered degrees of
freedom. On the other hand, such a decoupling may be beneficial. For example, having two sets of data for the discrete solution has been
used to design a smoothness indicator based on the difference between these two solutions; see \cite{ChKuMi_proc}. This smoothness indicator
can be used to develop different adaptation strategies, which may substantially enhance the resolution achieved by the DF-FV method. In
addition, we have been working on applications of the proposed framework to other problems, such as compressible multifluid flows, whose
investigation is left for an upcoming paper.

\smallskip
\noindent
$\bullet$ The 2-D DF-FV method relies on a different set of degrees of freedom from those used in the AF approach, as the 2-D DF-FV method
does not consider degrees of freedom at cell nodes.

\smallskip
We would also like to emphasize that the main difficulty in designing higher-order extensions of the proposed DF-FV methods is related to
developing a higher-order post-processing. Such post-processing should be based on a uniformly accurate reconstruction/interpolation. One
can use, for example, CWENO reconstructions/interpolations (see \cite{CPSV,CSV,SemVis} and references therein), and we plan to explore this
possibility in our future work.

\subsection*{Acknowledgment}
R. Abgrall was partialy supported by SNSF grant 200020\_204917 ``Structure preserving and fast methods for hyperbolic systems of
conservation laws''. The work of A. Chertock was supported in part by NSF grant DMS-2208438. The work of A. Kurganov was supported in part
by NSFC grants 12171226 and W2431004. The work of L. Micalizzi was supported in part by the LeRoy B. Martin, Jr. Distinguished Professorship
Foundation.

\noindent \section*{In memoriam}

\noindent This paper is dedicated to the memory of Prof. Arturo Hidalgo L\'opez
($^*$July 03\textsuperscript{rd} 1966 - $\dagger$August 26\textsuperscript{th} 2024) of the Universidad Politécnica de Madrid, organizer of HONOM 2019 and active participant in many other editions of HONOM.
Our thoughts and wishes go to his wife, Lourdes, and his sister, Mar\'ia Jes\'us, whom he left behind.


\end{document}